\documentclass[12pt,a4paper,dvipsnames]{amsart}
\usepackage[utf8]{inputenc}
\usepackage{tabularx}
\usepackage[table]{xcolor}
\usepackage{amsmath, nicefrac, amsthm, verbatim, amsfonts, amssymb, xcolor, bbm}
\usepackage{graphics, xspace, enumerate, cite}
\usepackage[a4paper,margin=3cm]{geometry}

\usepackage[bb=boondox]{mathalfa}
\usepackage{tikz}
\usepackage{graphicx}
\usepackage[colorlinks=true,citecolor=green,urlcolor=green,linkcolor=green,bookmarksopen=true,unicode=true,pdffitwindow=true]{hyperref}
\usepackage[english]{babel}
\usepackage[languagenames,fixlanguage]{babelbib}

\usepackage[labelformat=simple]{subcaption}

\makeatletter
\@namedef{subjclassname@2020}{\textup{2020} Mathematics Subject Classification}
\makeatother

\theoremstyle{plain}
\newtheorem{theorem}{Theorem}[section]
\newtheorem{corollary}[theorem]{Corollary}
\newtheorem{lemma}[theorem]{Lemma}

\theoremstyle{definition}
\newtheorem{remark}[theorem]{Remark}

\newtheorem{definition}[theorem]{Definition}

\usepackage{textcomp}

\newcommand{\bbN}{\mathbb{N}}

\newcommand{\bbR}{\mathbb{R}}

\newcommand{\floor}[1]{\lfloor #1 \rfloor}
\newcommand{\FLOOR}[1]{\left\lfloor #1 \right\rfloor}

\newcommand{\CEIL}[1]{\left\lceil #1 \right\rceil}

\newcommand{\FJADEF}[3]{{#1}:{#2}\to{#3}}

\numberwithin{equation}{section}

\definecolor{Maroon}{RGB}{140,10,0}

\hypersetup{
	colorlinks,
	linkcolor={Maroon},
	citecolor={Maroon},
	urlcolor={Maroon}
}

\newcommand{\e}{\ensuremath{\text{\textopenbullet}}}
\newcommand{\f}{\ensuremath{\text{\textbullet}}}

\title{Complexity Function of Jammed Configurations of Rydberg Atoms}

\author[T.\ Do\v{s}li\'{c}]{Tomislav\ Do\v{s}li\'{c}}
\address[Tomislav\ Do\v{s}li\'{c}]{Department of Mathematics\\
	Faculty of Civil Engineering\\
	University of Zagreb\\
	Zagreb\\
	Croatia \\ and
	Faculty of Information Studies \\
	Novo Mesto \\
	Slovenia}
\email{tomislav.doslic@grad.unizg.hr}

\author[M.\ Puljiz]{Mate\ Puljiz}
\address[Mate Puljiz]{Department of Applied Mathematics\\
	Faculty of Electrical Engineering and Computing\\
	University of Zagreb\\
	Zagreb\\
	Croatia}
\email{mate.puljiz@fer.hr}

\author[S.\ \v{S}ebek]{Stjepan\ \v{S}ebek}
\address[Stjepan\ \v{S}ebek]{Department of Applied Mathematics\\
	Faculty of Electrical Engineering and Computing\\
	University of Zagreb\\
	Zagreb\\
	Croatia}
\email{stjepan.sebek@fer.hr}

\author[J.\ \v{Z}ubrini\'{c}]{Josip\ \v{Z}ubrini\'{c}}
\address[Josip\ \v{Z}ubrini\'{c}]{Department of Applied Mathematics\\
	Faculty of Electrical Engineering and Computing\\
	University of Zagreb\\
	Zagreb\\
	Croatia}
\email{josip.zubrinic@fer.hr}


\begin{document}

\begin{abstract}
In this article, we determine the complexity function (configurational entropy) of jammed configurations of Rydberg atoms on a one-dimensional lattice. Our method consists of providing asymptotics for the number of jammed configurations determined by direct combinatorial reasoning. In this way we reduce the computation of complexity to solving a constrained optimization problem for the Shannon's entropy function. We show that the complexity can be expressed explicitly in terms of the root of a certain polynomial of degree $b$, where $b$ is the so-called blockade range of a Rydberg atom. Our results are put in a relation with the model of irreversible deposition of $k$-mers on a one-dimensional lattice.

\end{abstract}

\maketitle


%
%
%
%
\section{Introduction}

\emph{Rydberg atom} is a name given to an atom which has been excited into
a high energy level so that one of its electrons is able to travel much
farther from the nucleus than usual (up to $10^6$ times more, see \cite{Dunning_Killian_2021}).
In physics community, Rydberg atoms have been intensely studied and have
become a testing ground for various quantum mechanical problems in quantum
information processing, quantum computation and quantum simulation
\cite{36_saffman2010quantum}. See \cite{gallagher_1994} for a comprehensive
description of the physics of Rydberg atoms and their remarkable properties.
Due to their large size, Rydberg atoms can exhibit very large electric dipole
moments which results in strong interactions between two close Rydberg atoms.
This causes a blockage effect that prohibits the excitation of an atom
located close to an atom that is already excited to a Rydberg state
\cite{37_jaksch2000fast, 38_liebisch2005atom, 39_pohl2010dynamical,
40_viteau2012cooperative, 41_hofmann2013sub, 42_bernien2017probing}.
The simplest setting for studying Rydberg atoms and their blockage effect
is on a finite one-dimensional lattice. In this setting, each atom occupies
one site and each two excited atoms are at least $b\ge 1$ sites apart. The
positive integer $b$ is referred to as the \emph{blockade range} of the
model. We will be interested in maximal (or \emph{jammed}) configurations
where no further atoms can be excited. Note that in such a configuration
each two excited atoms are at most $2b$ sites apart. In physics literature, jammed states in similar deposition models have the interpretation of metastable states at low
enough temperature and/or high enough density, and are referred to as
valleys, pure states, quasi-states, and inherent structures
\cite{1_thouless1977solution, 2_kirkpatrick1978infinite, 3_mezard1987spin,
4_gotze1992relaxation, 5_biroli2000inherent, 6_debenedetti2001supercooled,
7_berthier2011theoretical}.

These kinds of models are usually studied in the context of random sequential
adsorption (RSA). Initially all atoms are in the ground state, and are
excited sequentially, at random, until a jammed configuration is reached.
Of most interest is the \emph{jamming limit}, which is defined as the
expected density of excited atoms in the jammed configuration. This dynamical
version of the problem was already studied in literature. In
\cite[\S IV]{46_krapivsky2020large} it was found that the jamming limit is
\begin{equation*}
\rho_{\infty}^{b\text{-Ryd}} = \int_0^1 \exp\left[-2\sum_{j=1}^{b} \frac{1-y^j}{j} \right] dy.
\end{equation*}
The jamming limit was also computed for an equivalent model of deposition of
linear polymers ($k$-mers) in \cite[\S 7.1]{30_krapivsky2010kinetic}
\begin{align*}
\rho_{\infty}^{k\text{-mer}} = k\int_0^\infty \exp\left[-u-2\sum_{j=1}^{k-1} \frac{1-e^{-ju}}{j} \right]du = k\int_0^1 \exp\left[-2\sum_{j=1}^{k-1} \frac{1-y^j}{j} \right]dy.
\end{align*}
The equivalence of models is reflected in the fact that
$\rho_{\infty}^{k\text{-mer}} = k\cdot\rho_{\infty}^{b\text{-Ryd}}$ for $b=k-1$.

In this context, the jammed configurations are sometimes called attractors,
and in the present article, we will be interested in the number of those
attractors and in details of their structure, above all in their density.

It is known that in similar models, the number of different jammed
configurations with density $0\le \rho \le 1$ tends to grow exponentially
with the length $L$ of configuration, see \cite{10_cornell1991domain,
11_de2003metastable, 12_derrida1986metastable, 13_masui1989metastable,
14_fredrickson1984kinetic, 15_jackle1991hierarchically, 16_sollich1999glassy,
17_crisanti2000inherent, 18_dean2001tapping, 19_dean2001steady,
20_lefevre2001tapping, 21_prados2001analytical, 22_de2002jamming,
23_palmer1985low, 24_elskens1987aggregation, 25_privman1992exact,
26_lin1993exact, 27_krapivsky1994kinetic}. Denoting this number by
$J_L(\rho)$, it is common to describe it using the so-called complexity
function (also called configurational entropy) $f(\rho)$ for which it holds
that $J_L(\rho)\sim e^{Lf(\rho)}$. It turns out (see e.g.\ Figure
\ref{fig:delta_Rydberg_separate}) that the density $\rho_\star^{b\text{-Ryd}}$
maximizing the complexity function is slightly different than the expected
density (jamming limit) of the dynamical model. This falsifies the flatness
hypothesis formulated by Edwards, see \cite{35_baule2018edwards} for a
recent review. Note that $\rho_\star^{b\text{-Ryd}}$ is the limit
(as $L$ tends to infinity) of the most probable densities in the equilibrium
models that assign equal probabilities to all jammed configurations.

Our main goal is to compute the complexity $f(\rho)$ of jammed configurations of Rydberg atoms using direct combinatorial reasoning. The problem reduces to solving a constrained optimization problem for the Shannon's entropy function. We show that the complexity function can be expressed explicitly in terms of the root of a certain polynomial of degree $b$. This work has been carried out simultaneously with \cite{KL}. The authors there introduce a novel method for determining the same complexity function. Their method is inspired by the theory of renewal processes.

\bigskip

The described model of Rydberg atoms on a one-dimensional lattice is equivalent to a number of other models already present in the literature. The case $b=1$ is the famous model introduced by Flory \cite{57_flory1939intramolecular} describing the mechanism of vinyl polymerization. This is in turn essentially equivalent to the Page-R\'enyi car parking problem \cite{Page, GerinPRparkingHAL} (which is a discrete version of the famous model introduced by R\'enyi in \cite{58_renyi1958one}) describing the jammed configurations of cars of length $2$. The equivalence is obtained by replacing each excited atom with a car taking up both the atom's and its right neighbor's site. Clearly, this only works for configurations not having an excited atom at the rightmost site. This means that the total number of jammed configurations is actually different in these two models, but only up to a constant factor, which does not affect the shape of the complexity function of these models. In chemistry, this model appears in the context of the irreversible deposition of dimers \cite{57_flory1939intramolecular}, and in graph theory, the jammed configurations correspond to maximal matchings in a path graph, see \cite{DoslicZubac}.

Similarly, the general case $b>1$ corresponds to irreversible deposition of $k$-mers ($k=b+1$) in a linear polymer of length $L$. The equivalence (again, up to a constant factor) is obtained by replacing each excited atom with a polymer taking up $b+1$ consecutive sites, starting from the atom's position, see Figure \ref{fig:RydkmerEq}. In this, and all the following figures, bullets ($\f$) represent Rydberg atoms (in the Rydberg model) or occupied sites (in the $k$-mer deposition model), while empty bullets ($\e$) represent neutral atoms (in the Rydberg model) or vacant sites (in the $k$-mer deposition model). Notice that the gaps between adjacent $k$-mers in jammed configurations of this deposition model are of size at most $k-1$. This equivalence allows us to easily transfer our results on Rydberg atoms to the setting of $k$-mer deposition model. The problem of irreversible deposition of $k$-mers was extensively studied in the literature, see \cite{28_evans1993random,29_talbot2000car,30_krapivsky2010kinetic,46_krapivsky2020large,54_gonzalez1974cooperative,55_bartelt1993car,56_bonnier1994pair}.
\begin{figure}
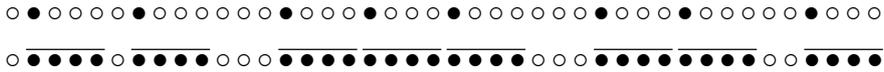

	\begin{gather*}
		\e\,\f\,\e\,\e\,\e\,\e\,\f\,\e\,\e\,\e\,\e\,\e\,\e\,\f\,\e\,\e\,\e\,\f\,\e\,\e\,\e\,\f\,\e\,\e\,\e\,\e\,\e\,\e\,\f\,\e\,\e\,\e\,\f\,\e\,\e\,\e\,\e\,\e\,\f\,\e\,\e\,\e\\
		\e\,\overline{\f\,\f\,\f\,\f}\,\e\,\overline{\f\,\f\,\f\,\f}\,\e\,\e\,\e\,\overline{\f\,\f\,\f\,\f}\,\overline{\f\,\f\,\f\,\f}\,\overline{\f\,\f\,\f\,\f}\,\e\,\e\,\e\,\overline{\f\,\f\,\f\,\f}\,\overline{\f\,\f\,\f\,\f}\,\e\,\e\,\overline{\f\,\f\,\f\,\f}
	\end{gather*}
	\caption{A jammed configuration of Rydberg atom model with blockade range $b$ and the corresponding jammed configuration of the $k$-mer deposition model when $b=3$, $k=4$.}\label{fig:RydkmerEq}
\end{figure}

In graph theory, the $k$-mer deposition model is equivalent to $P_k$-packings of a path graph $P_L$, and jammed configurations in the former correspond to maximal packings in the latter. The maximal $P_k$-packings of $P_L$ were previously studied in \cite{Doslic}.

Another equivalent formulation of the Rydberg atom model appeared recently in \cite[\S3.2.1]{DPSZ} where the authors of the present paper considered the settlement model consisting of $k$-story buildings on a one-dimensional tract of land. The tract of land is oriented east-west and each story of each building has to receive the sunlight from both east and west.

\bigskip

The rest of the paper is organized as follows. In Section
\ref{sectionjammedconfigurations} we calculate the asymptotics for the number
of jammed configurations in the model of Rydberg atoms, which is expressed
in terms of the maximum of the Shannon's entropy function over a certain
finite set determined by the constraints of the model. In Section
\ref{sectioncomplexity} we use these results in order to obtain the formula
for the associated complexity function. We derive the expression for the
complexity $f(\rho)$ which, for a chosen density $\rho$, depends explicitly
on a positive root of a certain polynomial whose degree coincides with the
blockade range of the model. Further on, in Section
\ref{sectioncomplexitykmers}, we put our findings in relation with the model
for the deposition of $k$-mers on the linear polymer and draw conclusions
from the obtained results. There, we also provide some results on the
qualitative properties of the maximizers of mentioned complexity functions,
for various blockade ranges $b$, and put them in comparison with their
counterparts in the theory of RSA. Finally, in Section \ref{sec:concluding} we recapitulate our findings and indicate several possible directions of future research.

\bigskip

\subsection*{Notation} We write $M_L\sim N_L$ if the two positive sequences $(M_L)_L$ and $(N_L)_L$ have the same growth, as $L\to \infty$, up to a sub-exponential factor, i.e.\ if $$\lim\limits_{L\to \infty} \dfrac{\ln M_L - \ln N_L}{L} = 0.$$

\section{Jammed configurations of Rydberg atoms}
\label{sectionjammedconfigurations}

As already stated in the introduction, the main goal of this paper is to compute the complexity function $f(\rho)$ of jammed configurations of Rydberg atoms. Crucial step towards obtaining a complexity function of such models in general is to inspect the set of all jammed configurations of a model. Each configuration is a binary 0/1 sequence which we sometimes interpret as a sequence of empty/occupied sites or, in Rydberg model, as neutral/excited atoms. The total number of all configurations of length $L$ in the model is denoted by $J_L$. The total number of configurations of length $L$ consisting of $N$ ones (occupied sites, excited atoms) is denoted by $J_{N,L}$. The \emph{density (saturation, coverage)} of any such configuration of length $L$ with $N$ ones is defined as $N/L\in[0,1]$.

In order to determine the complexity function, it is not enough to work only with $J_L$. We need to be more precise. We need to know the behavior of the number of different jammed configurations of length $L$, where precisely $N$ atoms are excited to the Rydberg state. The main result of this section (see Lemma \ref{lm:asympt}) provides asymptotics of the quantity $J_{N, L}$ for Rydberg atom model.

Let us first consider several concrete examples of jammed configurations of our model to get a better feeling of their possible shapes. Figure \ref{fig:examples_L=16} displays three different jammed configurations in the chain of $L = 16$ atoms, where the blockade range $b$ is equal to two, i.e.\ each two excited atoms are at least two sites apart.
\begin{figure}
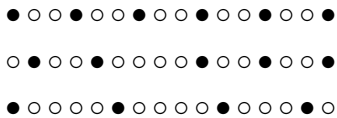

	\begin{gather*}
		\f\,\e\,\e\,\f\,\e\,\e\,\f\,\e\,\e\,\f\,\e\,\e\,\f\,\e\,\e\,\f\\
		\e\,\f\,\e\,\e\,\f\,\e\,\e\,\e\,\e\,\f\,\e\,\e\,\f\,\e\,\e\,\f\\
		\f\,\e\,\e\,\e\,\e\,\f\,\e\,\e\,\e\,\e\,\f\,\e\,\e\,\e\,\f\,\e
	\end{gather*}
	\caption{Three jammed configurations in the chain of $L = 16$ atoms with blockade range $b = 2$. The number of Rydberg atoms in these configurations is $N = 6, 5, 4$ (from top to bottom).}\label{fig:examples_L=16}
\end{figure}
Since Rydberg atoms in a jammed configuration are separated by clusters of empty sites whose length is at least $b$ (so that the constraint imposed by the blockage effect is satisfied), and at most $2b$ (since we can excite another atom in the middle of an empty range of size $2b + 1$, hence such a configuration would not be jammed), it is easy to see that it holds
\begin{equation}\label{eq:bounds_for_N}
	\CEIL{\frac{L}{2b+1}}\le N \le \CEIL{\frac{L}{b+1}},
\end{equation}
where $N$ is the number of excited atoms, $L$ is the length of the configuration, and $b$ is the blockade range. In the particular case of $L = 16$ and $b = 2$, this implies that $4 \le N \le 6$. Hence, Figure \ref{fig:examples_L=16} shows one jammed configuration for each possible value of $N$. Notice that relation \eqref{eq:bounds_for_N} implies that
\begin{equation}\label{eq:bounds_for_N/L}
	\frac{1}{2b+1} - \frac{1}{L} < \frac{N}{L}\le \frac{1}{b+1},
\end{equation}
and this in turn implies that in the limit, as $L\to\infty$, the density $\rho = N/L$, of Rydberg atoms in jammed configurations, lies within the bounds
\begin{equation}\label{eq:bounds_for_rho}
	\frac{1}{2b+1} \le \rho \le \frac{1}{b+1}.
\end{equation}

As a first result in the direction of better understanding the double sequence $J_{N, L}$ for Rydberg atom model, we provide the bivariate generating function for this sequence in the general case of blockade range $b \ge 1$.
\begin{lemma}\label{lm:BGF}
	The bivariate generating function of the sequence $J_{N, L}$ associated with jammed configurations of Rydberg atoms, when the blockade range is equal to $b$, is given by
	\begin{equation*}
		F_b(x, y) = \frac{(1-y)^2 + xy - xy^{b+1} - xy^{b+2} + xy^{2b+2}}{(1-y) (1 - y - xy^{b+1} + xy^{2b+2})},
	\end{equation*}
	where $x$ is a formal variable associated with the number of atoms excited to the Rydberg state, and $y$ is a formal variable associated with the length of the configuration.
\end{lemma}
\begin{proof}
	As already mentioned, configurations of Rydberg atoms can be represented as $0/1$ sequences. Due to the fact that we can determine whether the blockage effect has been taken into account, and whether the configuration represented with such a sequence is jammed, just by inspecting finite size patches of a given sequence, we can apply the so-called transfer matrix method (see \cite[\S 4.7]{Stanley} or
\cite[\S V]{FlajoletSedgewick}, and also \cite[\S 2--4]{SymbDynCoding}). This is a well known method for counting words of a regular language. Since Rydberg atoms in a jammed configuration are separated with at least $b$, and at most $2b$ neutral atoms, every jammed configuration will be composed of blocks that start with a Rydberg atom and then have a cluster of neutral atoms of length between $b$ and $2b$. Such blocks are displayed in Figure \ref{fig:building_blocks}.
\begin{figure}
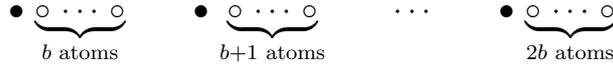

	\begin{gather*}
		\f\, \underbrace{\e\, \cdots \, \e}_{b \textnormal{ atoms}} \qquad \f\, \underbrace{\e\, \cdots \, \e}_{b+1 \textnormal{ atoms}} \qquad \cdots \qquad \f\, \underbrace{\e\, \cdots \, \e}_{2b \textnormal{ atoms}}
	\end{gather*}
	\caption{Building blocks of jammed configurations of Rydberg atoms with blockade range $b$.}\label{fig:building_blocks}
\end{figure}
These building blocks are encoded with the polynomial
	\begin{equation*}
		p_b(x, y) = xy^{b+1} + xy^{b+2} + \cdots + xy^{2b+1}.
	\end{equation*}
	Now we only need to take care of the beginning and the end of jammed configurations. Notice that in front of the first block we can have some neutral atoms. More precisely, the number of neutral atoms that can appear at the left end of the jammed configuration is between $0$ and $b$. These starting blocks are encoded with the polynomial
	\begin{equation*}
		s_b(x, y) = 1 + y + y^2 + \cdots + y^{b}.
	\end{equation*}
	Similarly, after the last block from the set of blocks shown in Figure \ref{fig:building_blocks} (if there are any, i.e.\ if we want to have more than just one atom in the Rydberg state), we need to have a block that again starts with a Rydberg atom, and then has a cluster of neutral atoms of length between $0$ and $b$. These ending blocks are encoded with the polynomial
	\begin{equation*}
		e_b(x, y) = xy + xy^2 + \cdots + xy^{b+1}.
	\end{equation*}
	Notice that each of the blocks shown in Figure \ref{fig:building_blocks} can be glued to any other block listed in this figure. This implies that we do not even need to work with powers of the transfer matrix, but we can directly take powers of the polynomial $p_b(x, y)$ in order to obtain the desired bivariate generating function. A simple calculation gives
	\begin{align*}
		F_b(x, y)
		& = 1 + \sum_{n = 0}^{\infty} s_b(x, y) \cdot p_b(x, y)^n \cdot e_b(x, y) \\
		& = 1 + \frac{s_b(x, y) \cdot e_b(x,y)}{1 - p_b(x, y)} \\
		& = \frac{(1-y)^2 + xy - xy^{b+1} - xy^{b+2} + xy^{2b+2}}{(1-y) (1 - y - xy^{b+1} + xy^{2b+2})}.
	\end{align*}
\end{proof}
\begin{remark}
By using the same technique, we can easily compute the bivariate generating function enumerating the number of jammed configurations of prescribed length, and with some fixed number of occupied sites, in the $k$-mer deposition model. The building blocks here are composed of a cluster of $k$ consecutive sites occupied by a single $k$-mer, followed by a cluster of empty sites of length between $0$ and $k-1$ (see Figure \ref{fig:building_blocks_k-mers}).
	\begin{figure}
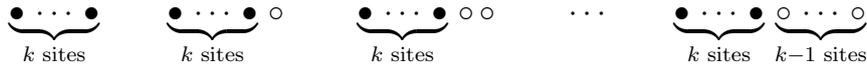

		\begin{gather*}
			\underbrace{\f\, \cdots \, \f}_{k \textnormal{ sites}} \qquad \underbrace{\f\, \cdots \, \f}_{k \textnormal{ sites}}\,\e\, \qquad \underbrace{\f\, \cdots \, \f}_{k \textnormal{ sites}}\,\e\,\e\, \qquad \cdots \qquad \underbrace{\f\, \cdots \, \f}_{k \textnormal{ sites}}\, \underbrace{\e\, \cdots \, \e}_{k-1 \textnormal{ sites}}
		\end{gather*}
		\caption{Building blocks of jammed configurations of $k$-mer deposition model.}\label{fig:building_blocks_k-mers}
	\end{figure}
	These building blocks are encoded with a polynomial
	\begin{equation*}
		p_k(x, y) = x^ky^k + x^ky^{k+1} + \cdots + x^ky^{2k-1},
	\end{equation*}
where $x$ is again a formal variable associated with the number of occupied sites, and $y$ is a formal variable associated with the length of a configuration. Similarly as in the case of the Rydberg atom model, at the left end of a jammed configuration, we can have a cluster of vacant sites of length between $0$ and $k-1$. These starting blocks are encoded with the polynomial
	\begin{equation*}
		s_k(x, y) = 1 + y + y^2 + \cdots + y^{k-1}.
	\end{equation*}
It is clear that we can end a jammed configuration with any of the building blocks shown in Figure \ref{fig:building_blocks_k-mers}, so we can set $e_k(x, y) = 1$. Using again the fact that each of the blocks from Figure \ref{fig:building_blocks_k-mers} can be glued to any other block listed in that figure, we can work directly with powers of the polynomial $p_k(x, y)$ to obtain
	\begin{equation}\label{eq:BGF_k-mers}
		F_k(x, y) = \sum_{n = 0}^{\infty} a_k(x, y) \cdot p_k(x, y)^n = \frac{a_k(x, y)}{1 - p_k(x, y)} = \frac{1 - y^k}{1 - y - x^ky^k + x^ky^{2k}}.
	\end{equation}
	Notice that we are not adding $1$ to the bivariate generating function in \eqref{eq:BGF_k-mers}. The reason is that starting with a cluster of $0$ vacant sites and setting $n = 0$ already counts the empty configuration.
\end{remark}
The sequence $J_{N, L}$ has already been studied in the literature, but in the context of maximal $P_k$-packings of a path graph $P_L$ (see \cite{Doslic}). The bivariate generating function enumerating the total number of maximal $k$-packings in $P_L$, with exactly $N$ copies of $P_k$, is given in \cite[Corollary 2.4]{Doslic}, and the only difference between that bivariate generating function and the one we obtained in \eqref{eq:BGF_k-mers}, is that $x$ is not raised to power $k$. The reason is that the author in \cite{Doslic} is interested in the number of copies of $P_k$ (i.e.\ the number of deposited $k$-mers) in jammed configurations, and we are interested in the total number of sites occupied by those deposited $k$-mers. The bivariate generating function from \eqref{eq:BGF_k-mers} is also obtained in \cite[formula (5.3)]{KL}, where authors use a novel approach inspired by the theory of renewal processes. Using the same technique, they also obtain the bivariate generating function which coincides with the one we obtained in Lemma \ref{lm:BGF}, which enumerates the total number of jammed configurations of length $L$ of Rydberg atoms with blockade range $b$, with precisely $N$ excited atoms (see \cite[formula (6.5)]{KL}).

It is easy to see from the bivariate generating function from Lemma \ref{lm:BGF} that, for $b = 2$, $J_{16} = 96$ (i.e.\ there are $96$ jammed configurations in the chain of $L = 16$ atoms, when the blockade range is $b = 2$). Out of those $96$ jammed configurations, $45$ of them have $4$ Rydberg atoms ($J_{4, 16} = 45$), $50$ of them have $5$ Rydberg atoms ($J_{5, 16} = 50$), and only one has $6$ Rydberg atoms ($J_{6, 16} = 1$). This particular one is exactly the first jammed configuration shown in Figure \ref{fig:examples_L=16}.

We could now proceed like the authors in \cite{KL} and use the bivariate generating function developed in Lemma \ref{lm:BGF} to obtain the complexity function of jammed configurations of Rydberg atoms by means of the Legendre transform. However, we will use a direct combinatorial argument. To this end, we introduce a slightly different way of counting jammed configurations in the Rydberg model with blockade range $b$, than the one introduced in Lemma \ref{lm:BGF}. Denote with $B$ the block of $b+1$ adjacent atoms where only the first one is excited to the Rydberg state (see Figure \ref{fig:def_of_B}).
\begin{figure}
	\begin{gather*}
		B = \f\, \underbrace{\e\,\e\, \cdots \, \e}_{b \textnormal{ atoms}}
	\end{gather*}
	\caption{Block consisting of $b+1$ adjacent atoms where only the first one is excited to the Rydberg state.}\label{fig:def_of_B}
\end{figure}
Using again the fact that each two Rydberg atoms have at least $b$ and at most $2b$ neutral atoms separating them, it is clear that every jammed configuration consists of blocks $B$ separated by clusters of neutral atoms of length $0 \le a \le b$ (see Figure \ref{fig:jammed_conf}).
\begin{figure}
	\begin{gather*}
		\underbrace{\e\, \cdots \, \e}_{a_1 \textnormal{ atoms}}\, B\, \underbrace{\e\, \cdots \, \e}_{a_2 \textnormal{ atoms}}\, B\, \underbrace{\e\, \cdots \, \e}_{a_3 \textnormal{ atoms}}\, B\,\cdots B\,\underbrace{\e\, \cdots \, \e}_{a_{N} \textnormal{ atoms}}\, B\
	\end{gather*}
	\caption{The shape of jammed configurations in Rydberg model with blockade range $b$ and exactly $N$ Rydberg atoms, ending with a block $B$ (displayed in Figure \ref{fig:def_of_B}). Gaps between blocks $B$, and in front of the first block $B$, consist of neutral atoms and are of length $0 \le a_i \le b$.}\label{fig:jammed_conf}
\end{figure}
Denote by $M_a$ the number of gaps with $a$ neutral atoms. The total number of jammed configurations of the shape shown in Figure \ref{fig:jammed_conf}, with $L$ atoms in total, out of which precisely $N$ atoms are excited to the Rydberg state, is given as
\begin{equation}\label{eq:multinom_coef}
	\binom{N}{M_0, M_1, \ldots, M_b} = \frac{N!}{\prod_{0 \le a \le b} M_a!},
\end{equation}
with $M_a$ satisfying
\begin{align}\label{eq:N_gaps}
	\sum_{a = 0}^b M_a & = N,\\\label{eq:L_atoms}
	\sum_{a = 0}^b aM_a & = L - (b + 1)N.
\end{align}
The constraint \eqref{eq:N_gaps} expresses that the total number of gaps is $N$. Notice that we have $N$ blocks $B$ (since we want to have precisely $N$ Rydberg atoms), and that gaps of size $0 \le a \le b$ can be added in front of the first block $B$, and between each two blocks $B$. The constraint \eqref{eq:L_atoms} implies that the total number of neutral atoms is $L - N$. Clearly we need $L - N$ neutral atoms in addition to $N$ Rydberg atoms to have a configuration of length $L$. Equation \eqref{eq:multinom_coef} accounts for the jammed configurations ending precisely on $B$. There are also jammed configurations where the last block $B$ is truncated, and there are only $0\le c<b$ neutral atoms after the last atom excited to the Rydberg state. The contribution of such jammed configurations to the value of $J_{N, L}$ is comparable to \eqref{eq:multinom_coef}, but since complexity function ignores sub-exponential factors, it  suffices to determine the asymptotics of the sum
\begin{equation}\label{eq:24}
	J_{N,L} \sim \sum_{(M_0,M_1,\dots,M_b) \in R_{N,L}} \binom{N}{M_0, M_1, \dots, M_b},
\end{equation}
where
\begin{multline}\label{eq:def_of_RNL}
	R_{N,L} = \{(M_0,M_1,\dots,M_b) \in \bbN_0^{b+1}:  M_0+M_1+\dots+M_b = N \text{ and }\\  M_1 + 2M_2+\dots+bM_b = L - (b+1)N\}.
\end{multline}
We write $H$ for the Shannon's entropy function given as
\begin{equation}\label{eq:Shannon_entropy}
	H(p_0,p_1,\dots,p_b) = -\sum_{i=0}^{b} p_i \ln p_i,
\end{equation}
where $p_i\ge0$, for $0\le i\le b$, and $p_0+p_1+\dots+p_b=1$.
\begin{remark}
	In case $p_i=0$ for some $i$, we set $0\cdot \ln 0 =0$.
\end{remark}

The following lemma is the key result of this section, and it constitutes a crucial step in computing the complexity function of our model as it provides the asymptotics of $J_{N,L}$ in terms of the maximum of the entropy function.

\begin{lemma}\label{lm:asympt}
	$$
	J_{N,L} \sim \exp\left(L\cdot \max_{(M_0,M_1,\dots,M_b) \in R_{N,L}} \frac{N}{L}\cdot H\left(\frac{M_0}{N},\frac{M_1}{N},\dots,\frac{M_b}{N}\right)\right), \text{ as } L \to \infty.
	$$
	where the set $R_{N, L}$ is defined in \eqref{eq:def_of_RNL}, and the function $H$ is defined in \eqref{eq:Shannon_entropy}.
\end{lemma}
\begin{proof}
	Note that
	\begin{align*}
		\max_{(M_0, M_1, \ldots, M_b)\in R_{N, L}} \binom{N}{M_0, M_1, \ldots, M_b}
		& \le \sum_{(M_0, M_1, \dots, M_b)\in R_{N, L}} \binom{N}{M_0, M_1, \ldots, M_b} \\
		& \le |R_{N, L}| \max_{(M_0, M_1, \ldots, M_b)\in R_{N, L}} \binom{N}{M_0, M_1, \ldots, M_b}.
	\end{align*}
	As the number $|R_{N, L}|$ of terms in the sum is at most $(N+1)^{b+1}\le (L+1)^{b+1}$, which is polynomial in $L$, the sum, asymptotically, grows as its largest term. It is, therefore, enough to determine the asymptotics of
	$$J_{N,L} \sim \max_{(M_0,M_1,\ldots,M_b) \in R_{N,L}} \binom{N}{M_0, M_1, \ldots, M_b}, \text{ as } L \text{ (and } N)\to \infty.
	$$
	By following the proof of Lemma 2.2 in \cite{csiszar2004information} we can conclude that
	$$\binom{N+b}{b}^{-1}\frac{N^N}{{M_0}^{M_0} {M_1}^{M_1}\cdots{M_b}^{M_b}} \le \binom{N}{M_0, M_1, \dots, M_b} \le \frac{N^N}{{M_0}^{M_0} {M_1}^{M_1}\cdots{M_b}^{M_b}}.$$
	Note that in case any $M_a$ is zero, the expression $0^0$ is to be interpreted as $1$. Since $\binom{N+b}{b}$ is of polynomial growth, we get
	\begin{equation}\label{eq:Stirling}
	\begin{aligned}
	\binom{N}{M_0, M_1, \dots, M_b} &\sim \frac{N^N}{{M_0}^{M_0} {M_1}^{M_1}\cdots{M_b}^{M_b}} \\
	&= \left(\frac{N}{M_0}\right)^{M_0}\left(\frac{N}{M_1}\right)^{M_1}\cdots\left(\frac{N}{M_b}\right)^{M_b},
	\end{aligned}
	\end{equation}
	as $N\to \infty$.
	Note that $$\left(\frac{N}{M_0}\right)^{M_0}\left(\frac{N}{M_1}\right)^{M_1}\cdots\left(\frac{N}{M_b}\right)^{M_b} = \exp\left(N\cdot H\left(\frac{M_0}{N},\frac{M_1}{N},\dots,\frac{M_b}{N}\right)\right).$$
	Hence
	$$
	J_{N,L} \sim \max_{(M_0,M_1,\dots,M_b) \in R_{N,L}} \exp\left(N\cdot H\left(\frac{M_0}{N},\frac{M_1}{N},\dots,\frac{M_b}{N}\right)\right), \text{ as } L \to \infty,
	$$
	and consequentially
	$$
	J_{N,L} \sim \exp\left(L\cdot \max_{(M_0,M_1,\dots,M_b) \in R_{N,L}} \frac{N}{L}\cdot H\left(\frac{M_0}{N},\frac{M_1}{N},\dots,\frac{M_b}{N}\right)\right), \text{ as } L \to \infty,
	$$
	which is exactly what we wanted to prove.
\end{proof}
\begin{remark}
One could obtain the asymptotics in \eqref{eq:Stirling} from Stirling's approximation $N!\sim (N/e)^N$, as $N\to\infty$, where sub-exponential factors are ignored.
%
\end{remark}

\section{Complexity function of jammed configurations of Rydberg atoms}
\label{sectioncomplexity}
In this section we compute the \emph{complexity function}, sometimes referred to as \emph{configurational entropy}, of jammed configurations of Rydberg atoms. We first recall the definition of complexity function of a certain model.
\begin{definition}
For a fixed density $\rho \in [0,1]$, let $J_{\FLOOR{\rho L},L}$ denote the number of configurations of length $L$ with density $\FLOOR{\rho L}/L\approx\rho$. The complexity function $\FJADEF{f}{[0,1]}{\bbR}$ is then defined as
\begin{equation}\label{eq:cmplxDEF}
f(\rho) = \lim_{L\to \infty} \frac{\ln J_{\floor{\rho L},L}}{L},
\end{equation}
for each $\rho\in[0,1]$ for which this limit exists.
\end{definition}
\begin{remark}
	If the limit above does not exist for a certain $\rho$, one can still define (upper) complexity at that point by replacing $\lim$ in the definition with $\limsup$. And if there are no configurations with a certain density $\rho$, we still write $f(\rho)=0$.
\end{remark}
\begin{remark}
		This definition implies that the number of configurations with density $\FLOOR{\rho L}/L\approx\rho$ grows as $e^{Lf(\rho)}$ for large $L$.
\end{remark}
The guiding idea behind introducing the complexity function is to describe what portion of the total number of configurations take up configurations with a particular density. The problem is that, as $L$ grows to infinity, the actual proportions tend to the delta distribution concentrated on the `most probable' density $\rho_\star$.

As an example, the distribution of densities (the sum of digits divided by the
length) of binary sequences of length $L$ is a symmetric binomial
distribution re-scaled to the interval $[0,1]$. The limiting distribution is
then the delta distribution $\delta_{0.5}$ which is, essentially, the
consequence of the law of large numbers.

This convergence to a delta distribution results from the fact that the number
of configurations with a certain density grows exponentially with a rate that
depends on the density. For large $L$, the number of configurations with
density having the largest rate overtakes, in proportion, configurations having
any other density. The complexity function then quantifies the distribution of
all configurations with respect to their densities in a more refined way.

Another consequence of the fact that the number of configurations having density with the largest rate dominates, in proportion, any other density is that the total number of all configurations grows at the same exponential rate as the number of configurations having this `most probable density'. To be precise, if $\rho_\star$ denotes the density at which the complexity function $f$ attains its maximum and if $J_L$ is the total number of all configurations of length $L$, then $J_L \sim e^{Lf(\rho_\star)}$ for large $L$.

\begin{remark}\label{rem:maxIsWb}
	In Lemma \ref{lm:BGF} we derived the generating function for the sequence
	$J_{N,L}$ within the Rydberg atom model. Plugging $x=1$ into this
	generating function gives the generating function for $J_L$, the total
	number of configurations of length $L$ in Rydberg atom model
	\begin{align*}
	F_b(1, y) &= \frac{(1-y)^2 + y - y^{b+1} - y^{b+2} + y^{2b+2}}{(1-y) (1 -
	y - y^{b+1} + y^{2b+2})} \\
	&= \frac{1+y(1+y+\dots+y^{b})(1+y+\dots+y^{b-1})}{1-y^{b+1}(1+y+\dots+y^{b})}.
	\end{align*}
	From here, we can infer the asymptotics of $J_L$ for large $L$ by inspecting the roots of the polynomial $1-y^{b+1}(1+y+\dots+y^{b})$ in the denominator. More precisely, if $y_b$ is the root with the smallest modulus, then the logarithm of $w_b = |y_b|^{-1}$ gives the exponential growth rate of the sequence $J_L$
	$$J_L\sim w_b^{L} = e^{L \ln w_b}.$$
	The discussion in the previous paragraph now implies the relation $f(\rho_\star^{b\text{-Ryd}}) = \ln w_b.$
\end{remark}

The following theorem is the main result of this paper and provides an elegant expression for the complexity function of jammed configurations of Rydberg atoms $f(\rho)$ in terms of a root of a certain polynomial.
\begin{theorem}\label{tm:cmplx}
	The complexity function of jammed configurations of Rydberg atoms with blockade range $b\in\bbN$ is given as
	$$f(\rho)  = \begin{cases}
	\rho\left[-\ln \frac{1-z}{1-z^{b+1}} - \left(\frac{1}{\rho}-(b+1)\right)\ln z\right], &\text{ if } \frac{1}{2b+1}<\rho \le \frac{1}{b+1},\\
	0, &\text{ otherwise,}
	\end{cases}$$
	where $z\ge 0$ is a real root of the polynomial
	\begin{equation}\label{eq:PolyEq}
		p(z)=\sum_{i=0}^{b} \left(i+b+1-\frac{1}{\rho}\right)z^i
	\end{equation}
	for which the expression $f(\rho)$ is the largest.
\end{theorem}
\begin{remark}
	When $\frac{1}{2b+1}<\rho<\frac{1}{b+1}$ the leading coefficient of the polynomial $p(z)$ given in \eqref{eq:PolyEq} is positive, while the constant term is negative. This guaranties the existence of at least one positive real root $z>0$. If $\rho=\frac{1}{b+1}$, then $z=0$ is the root of $p(z)$ and the formula gives $f(\frac{1}{b+1})=0$.
\end{remark}
\begin{remark}
	Since \eqref{eq:PolyEq} is a polynomial of degree $b$, it is possible to find its roots explicitly for $b\le 4$ and numerically for $b>4$. The explicit expression for the complexity in case $b=1$ is
	$$f^{1\text{-Ryd}}(\rho) = \rho \ln \rho -(1-2\rho) \ln (1-2\rho) -(3\rho-1)\ln (3\rho-1),$$
	and for $b=2$
	\begin{multline*}
	f^{2\text{-Ryd}}(\rho) =(3\rho-1)\ln\frac{\sqrt{-44\rho^2+24\rho-3}-4\rho+1}{10\rho-2}-\\
	\rho\ln \frac{-350\rho^3+(25\rho^2-10\rho+1)\sqrt{-44\rho^2+24\rho-3}+215\rho^2-44\rho+3}{\rho^2\sqrt{-44\rho^2+24\rho-3}-134\rho^3+57\rho^2-6\rho}.
	\end{multline*}
	In the case $b = 1$, the function $f^{1\text{-Ryd}}(\rho)$ recovers the result from \cite[formula (7.20)]{30_krapivsky2010kinetic} and \cite[\S VII]{Krapivsky_2013}. The graphs of the complexity function of jammed configurations of Rydberg atoms with blockade range $1\le b\le 10$ are given in Figure \ref{fig:complexiyRydberg}. In that figure we also see that, for each $b$, the maximum of the complexity function matches $\ln w_b$, the growth rate of all jammed configurations. This was already discussed in Remark \ref{rem:maxIsWb}.
\end{remark}
\begin{figure}
	\includegraphics[width=\linewidth]{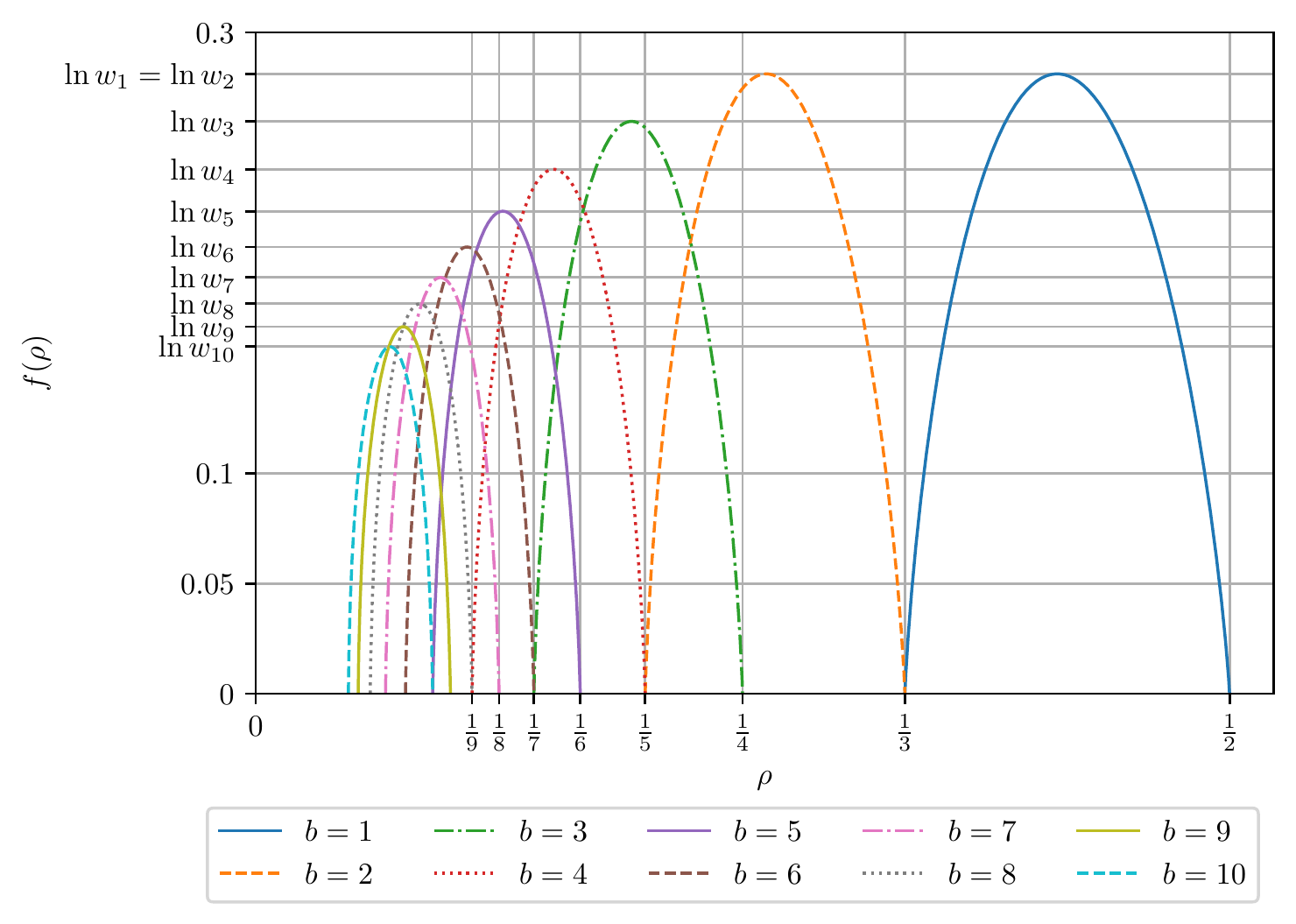}
	\caption{The complexity function of jammed configurations of Rydberg atoms with blockade range $1\le b\le 10$.}\label{fig:complexiyRydberg}
\end{figure}
\begin{proof}[Proof of Theorem \ref{tm:cmplx}]
	Recall that in \eqref{eq:bounds_for_N/L} we showed that
	$$\frac{1}{2b+1}-\frac{1}{L} < \frac{N}{L}\le\frac{1}{b+1},$$
	and therefore, there are no jammed configurations with densities $\rho>\frac{1}{b+1}$ nor with densities $\rho<\frac{1}{2b+1}$, for sufficiently large $L$. Thus, $f(\rho)=0$ when $\rho>\frac{1}{b+1}$ or $\rho<\frac{1}{2b+1}$. In case $\rho = \frac{1}{2b+1}$, it is not hard to see that the number of configurations $J_{\FLOOR{\frac{L}{2b+1}}, L}$ is
	$$J_{\FLOOR{\frac{L}{2b+1}}, L} =\begin{cases}
	1, &\text{if } (2b+1)\mid L,\\
	0, &\text{otherwise}.
	\end{cases}$$
	This implies $f(\frac{1}{2b+1})=0$ by the definition of complexity.

	In the remainder, we fix $\frac{1}{2b+1}<\rho \le \frac{1}{b+1}$. By Lemma \ref{lm:asympt}, and by using the definition of the complexity function \eqref{eq:cmplxDEF}, we have
\begin{equation}\label{eq:cmplx1}
	f(\rho) = \lim_{L\to \infty} \max_{(M_0,M_1,\dots,M_b) \in R_{N,L}} \frac{N}{L}\cdot H\left(\frac{M_0}{N},\frac{M_1}{N},\dots,\frac{M_b}{N}\right),
\end{equation}
where $N=\FLOOR{\rho L}$, provided that this limit exists.
By rewriting the constraint $(M_0,M_1,\dots,M_b) \in R_{N,L}$ as
\begin{gather*}\frac{M_0}{N}\ge 0, \frac{M_1}{N}\ge 0, \dots, \frac{M_b}{N}\ge 0 \\ \frac{M_0}{N}+\frac{M_1}{N}+\dots+\frac{M_b}{N} = 1 \\ \frac{M_1}{N} + 2\frac{M_2}{N}+\dots+b\frac{M_b}{N} = \frac{L}{N} - (b+1)\end{gather*}
and denoting $p_i = \frac{M_i}{N}\in\frac{1}{\FLOOR{\rho L}}\mathbb{Z}$, the complexity \eqref{eq:cmplx1} can be written as
\begin{equation}\label{eq:cmplexity}
	f(\rho) = \lim_{L\to \infty} \max_{(p_0, p_1, \dots, p_b)\in\frac{1}{\FLOOR{\rho L}} R_{\FLOOR{\rho L},L}}  \hat\rho H\left(p_0,p_1,\dots, p_b\right),
\end{equation}
where $\hat\rho = \hat{\rho}(L)=\frac{N}{L}=\frac{\FLOOR{\rho L}}{L}$. We claim that this limit exists and is equal to the maximum of the constrained optimization problem
\begin{equation}\label{eq:optproblem}
\max_{\substack{p_0, p_1, \dots, p_b\ge 0 \\ p_0+p_1+\dots+p_b = 1 \\ p_1 + 2p_2+\dots+bp_b = \frac{1}{\rho} - (b+1)}}  \rho H\left(p_0,p_1,\dots, p_b\right),
\end{equation}
where $p_i\in\bbR$ are no longer required to be fractions.

We argue as follows. Denote by $(p_0^\ast,p_1^\ast,\dots,p_b^\ast)$ the point at which the maximum in \eqref{eq:optproblem} is attained. For each $L\in\bbN$, let $(p_0(L),p_1(L),\dots, p_b(L))$ be the point at which maximum in \eqref{eq:cmplexity} is attained. Clearly,
$$\hat\rho H\left(p_0(L),p_1(L),\dots, p_b(L)\right) \le \hat\rho H(p_0^\ast,p_1^\ast,\dots,p_b^\ast)\le \rho H(p_0^\ast,p_1^\ast,\dots,p_b^\ast).$$
The first inequality follows by substituting $\hat\rho$ for $\rho$ in \eqref{eq:optproblem} and the fact that one is now optimizing over a larger set. The second inequality follows from $\hat\rho\le \rho$. Note that the right hand side no longer depends on $L$, and thus
$$\limsup_{L\to\infty} \hat\rho H\left(p_0(L),p_1(L),\dots, p_b(L)\right) \le \rho H(p_0^\ast,p_1^\ast,\dots,p_b^\ast).$$

Next, for each $L\in\bbN$, we consider the point $(t_0(L),t_1(L),\dots,t_b(L))\in \frac{1}{\FLOOR{\rho L}} R_{\FLOOR{\rho L},L}$, which is closest to the to the optimizer $(p_0^\ast,p_1^\ast,\dots,p_b^\ast)$. Note that, due to the density argument, $(t_0(L),t_1(L),\dots,t_b(L))\to (p_0^\ast,p_1^\ast,\dots,p_b^\ast)$ as $L\to\infty$. This, along with the continuity of $H$ and the fact that $\hat\rho\to\rho$ implies the lower bound
\begin{multline*}
\rho H(p_0^\ast,p_1^\ast,\dots,p_b^\ast) = \lim_{L\to\infty} \hat\rho H(t_0(L),t_1(L),\dots,t_b(L))\le \\
\le\liminf_{L\to\infty} \hat\rho H\left(p_0(L),p_1(L),\dots, p_b(L)\right).
\end{multline*}
Putting everything together completes the argument that the limit
$$f(\rho) = \lim_{L\to\infty} \hat\rho H\left(p_0(L),p_1(L),\dots, p_b(L)\right)$$
exists and that the complexity function is
$$f(\rho) =\rho H(p_0^\ast,p_1^\ast,\dots,p_b^\ast)= \max_{\substack{p_0, p_1, \dots, p_b\ge 0 \\ p_0+p_1+\dots+p_b = 1 \\ p_1 + 2p_2+\dots+bp_b = \frac{1}{\rho} - (b+1)}} \rho \cdot H\left(p_0,p_1,\dots, p_b\right).$$
In order to obtain the expression for complexity $f(\rho)$, it only remains to solve the constrained optimization problem \eqref{eq:optproblem}. We define the Lagrangian function
\begin{multline*}
	\mathcal{L}(p_0, \dots, p_b; \lambda, \mu) = \rho\cdot H(p_0,p_1,\dots,p_b) - \lambda (p_0+p_1+\dots+p_b-1) \\ - \mu (p_1+2p_2\dots+bp_b-\frac{1}{\rho}+(b+1)),
\end{multline*}
and find the stationary point by solving the system
\begin{equation}\label{eq: Lagrange}
\begin{split}
-\rho (\ln p_i +1)-\lambda - \mu i&=0 , \qquad\text{ for } i=0,1,\dots,b;\\
p_0+p_1+\dots+p_b&=1;\\
p_1+2p_2\dots+bp_b&=\frac{1}{\rho}-(b+1).
\end{split}
\end{equation}
By multiplying $i$-th of the first $(b+1)$ equations by $p_i$ and adding them together we get
$$-\rho \sum_{i=0}^b (p_i\ln p_i +p_i)-\lambda\sum_{i=0}^b p_i - \mu\sum_{i=0}^b i p_i=0,$$
and from here we obtain the expression for complexity in terms of the Lagrange multipliers $\lambda$ and $\mu$ which solve the system \eqref{eq: Lagrange}
\begin{equation}\label{eq:complexityLambdaMu}
f(\rho) = \rho H(p_0,p_1,\dots,p_b) = \rho + \lambda + \mu \left(\frac{1}{\rho}-(b+1)\right).
\end{equation}
Subtracting successive equations in \eqref{eq: Lagrange} we get
$$-\rho (\ln p_i - \ln p_{i-1})- \mu =0,$$
or equivalently
$$\frac{p_i}{p_{i-1}} = e^{-\mu/\rho}.$$
Therefore $p_i = p_0 e^{-\mu i/\rho}$, for $i=1,\dots,b$. From the very first equation in \eqref{eq: Lagrange} we get
$$p_0 =e^{-\lambda/\rho-1},$$
and the whole system \eqref{eq: Lagrange} now reduces to just two equations
\begin{align}\label{eq:system1}
e^{-\lambda/\rho-1}\sum_{i=0}^b e^{-\mu i/\rho}&=1;\\\label{eq:system2}
e^{-\lambda/\rho-1}\sum_{i=0}^b ie^{-\mu i/\rho}&=\frac{1}{\rho}-(b+1).
\end{align}
Setting $z=e^{-\mu/\rho}$, and eliminating $e^{-\lambda/\rho -1}$ from equations \eqref{eq:system1} and \eqref{eq:system2}, gives a single polynomial equation of degree $b$
\begin{equation}\label{eq:polyEquation}
bz^b+(b-1)z^{b-1}+\dots+2z^2+z=\left[\frac{1}{\rho}-(b+1)\right](z^b+z^{b-1}+\dots+z+1),
\end{equation}
which can be written as $p(z)=0$ where $p(z)$ is given in \eqref{eq:PolyEq}.

Now, in order to obtain the complexity, all we need is, for a fixed $\frac{1}{2b+1}<\rho<\frac{1}{b+1}$, to find a real root $z>0$ of the polynomial $p(z)$ for which the expression \eqref{eq:complexityLambdaMu} is the largest. The case $\rho=\frac{1}{b+1}$, which gives $z=0$, has to be treated separately. From relation $z=e^{-\mu/\rho}$ and equation \eqref{eq:system1} we have
\begin{equation}\label{eq:zadnja}
\begin{aligned}
	\mu &=-\rho \ln z;\\
	\lambda &=-\rho \left(1+\ln \frac{1-z}{1-z^{b+1}}\right).
\end{aligned}
\end{equation}
Plugging \eqref{eq:zadnja} into \eqref{eq:complexityLambdaMu}, gives the complexity expressed in terms of the root of $p(z)$
$$f(\rho)  = \rho\left[-\ln \frac{1-z}{1-z^{b+1}} - \left(\frac{1}{\rho}-(b+1)\right)\ln z\right].$$

Lastly, in case $\rho=\frac{1}{b+1}$, already from the last two equations in \eqref{eq: Lagrange} we can conclude $p_1=p_2=\dots=p_b=0$ and $p_0=1$. This immediately gives $f(\rho) = 0$ as $H(1,0,0,\dots,0)=0$, completing the proof.
\end{proof}

\begin{remark}
Using the standard summation formulas, we can rewrite \eqref{eq:polyEquation} as
\begin{equation}\label{eq:zPoly}
\frac{b z^{b+2} -(b+1)z^{b+1}+ z}{(1-z)^2}=\left[\frac{1}{\rho}-(b+1)\right]\frac{1-z^{b+1}}{1-z},
\end{equation}
or equivalently
\begin{equation*}
\left[(2b+1)-\frac{1}{\rho}\right] z^{b+2}
-\left[(2b+2)-\frac{1}{\rho}\right]z^{b+1} -\left[b-\frac{1}{\rho}\right]z+\left[(b+1)-\frac{1}{\rho}\right]=0.
\end{equation*}
\end{remark}

As discussed in the introduction, the complexity function is associated to \emph{equilibrium} (or \emph{static}) models of a certain phenomena and $\rho_\star$, the point at which the complexity function attains its maximum, is interpreted as the expected and most probable density observed in such a model. This value $\rho_\star$ is sometimes called the \emph{equilibrium density} of the model and Theorem \ref{tm:rhostar} below shows how to calculate it. A different (and perhaps more natural) way to look at Rydberg atom model is \emph{dynamically}, within the framework of random sequential adsorption (RSA). Initially neutral atoms are sequentially and at random excited (obeying the blockade range constraint) until the jammed configuration is reached. The expected density of the reached jammed configuration (the \emph{jamming limit}) in this dynamical version of the model, denoted by $\rho_{\infty}^{b\text{-Ryd}}$, was computed in \cite[\S IV]{46_krapivsky2020large}
$$
\rho_{\infty}^{b\text{-Ryd}} = \int_0^1 \exp\left[-2\sum_{j=1}^{b} \frac{1-y^j}{j} \right]\,dy.
$$
It is interesting to compare $\rho_{\star}^{b\text{-Ryd}}$ and $\rho_{\infty}^{b\text{-Ryd}}$ for different blockade ranges $b$.
\begin{figure}
	\includegraphics[width=.6\linewidth]{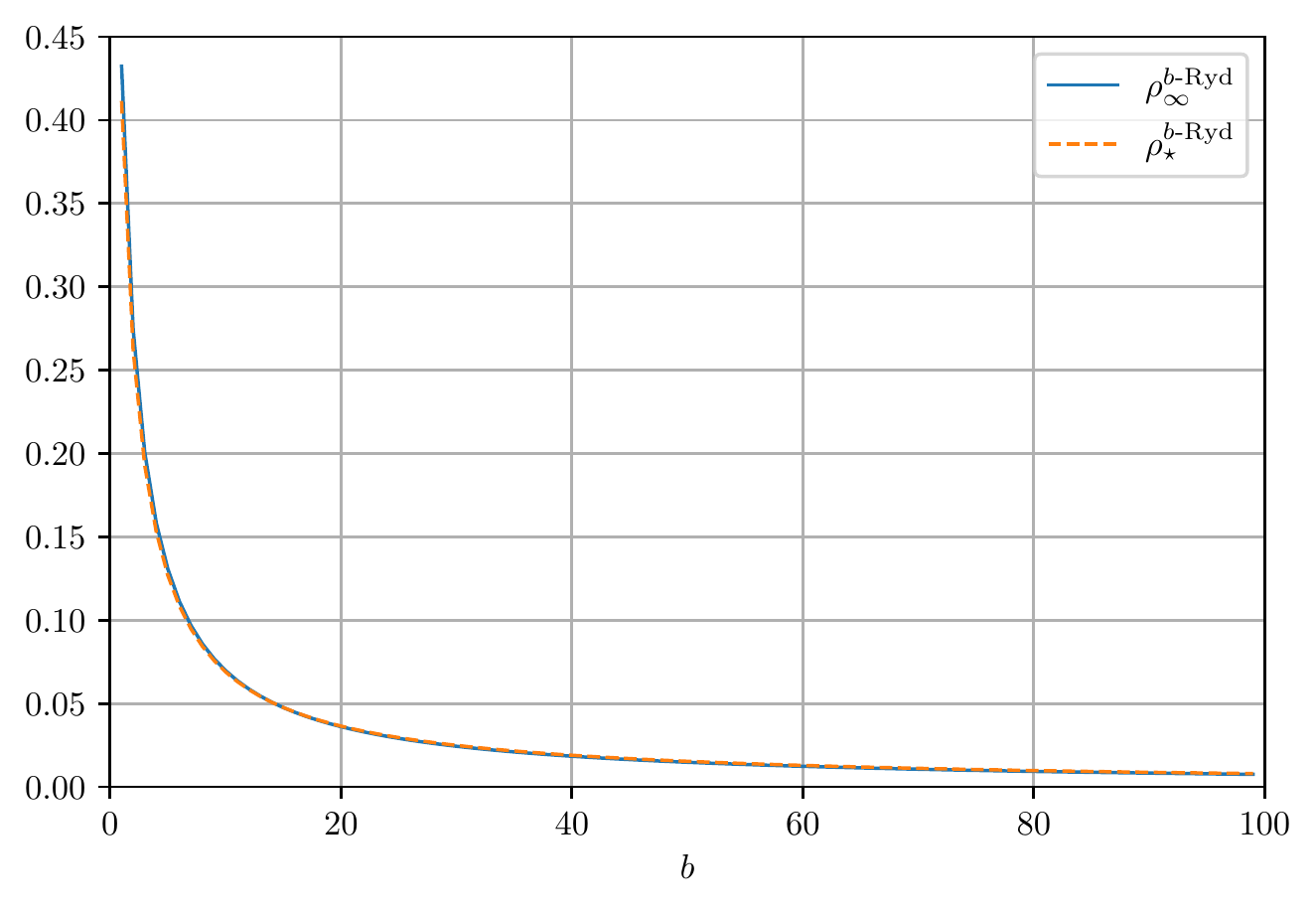}
	\caption{Comparison of $\rho_\star^{b\text{-Ryd}}$ and $\rho_\infty^{b\text{-Ryd}}$ for $1\le b \le 99$.}\label{fig:delta_Rydberg}
\end{figure}
Even though they are not the same, they seem to match quite nicely, see Figure \ref{fig:delta_Rydberg}. Additionally, as one would expect, they both tend to zero for large $b$. One can see their differences more clearly in Figure \ref{fig:delta_Rydberg_separate}.
\begin{figure}
	\begin{subfigure}{.48\linewidth}\centering
		\includegraphics[width=\linewidth]{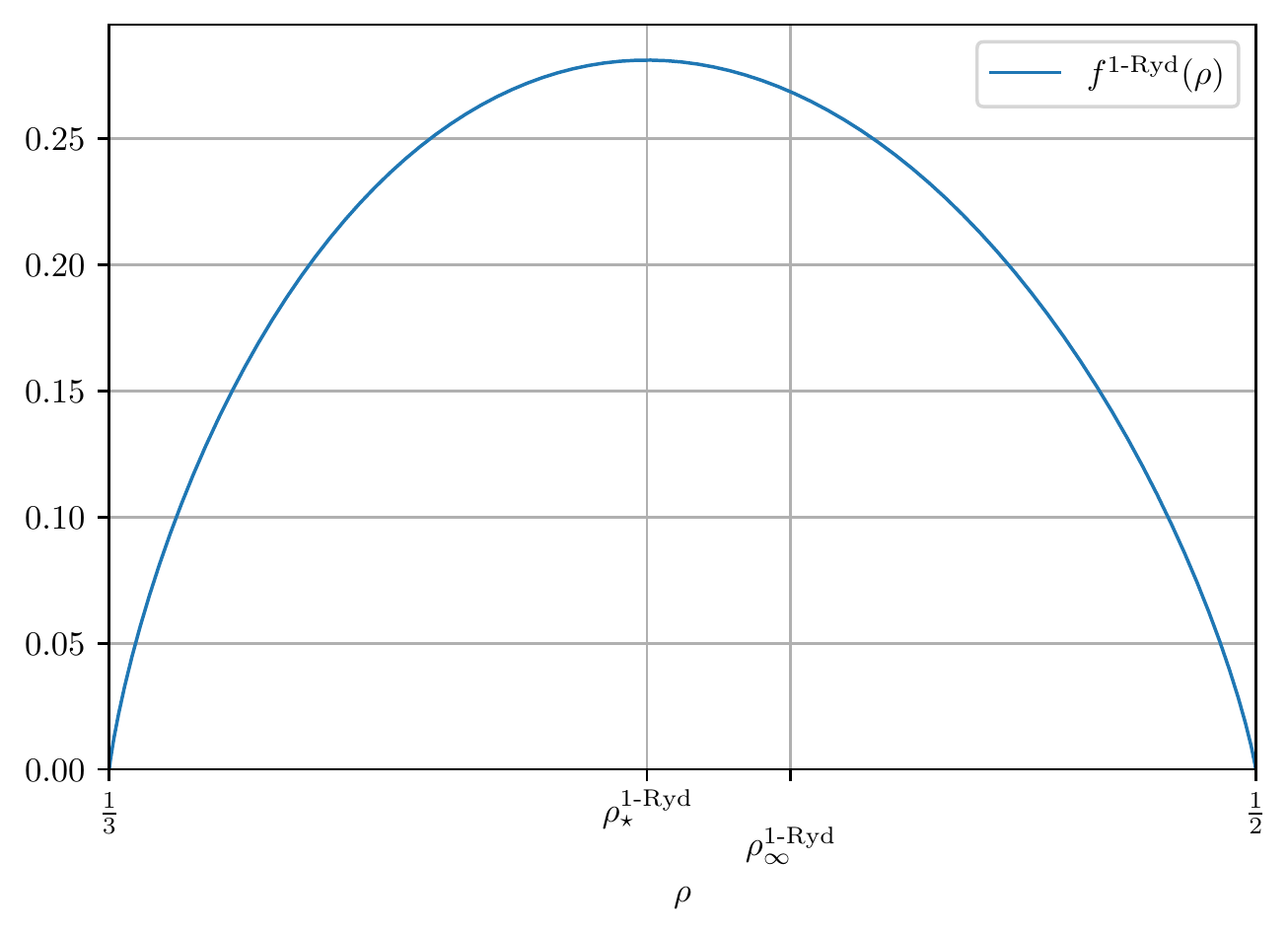}
	\end{subfigure}
	\begin{subfigure}{.48\linewidth}\centering
		\includegraphics[width=\linewidth]{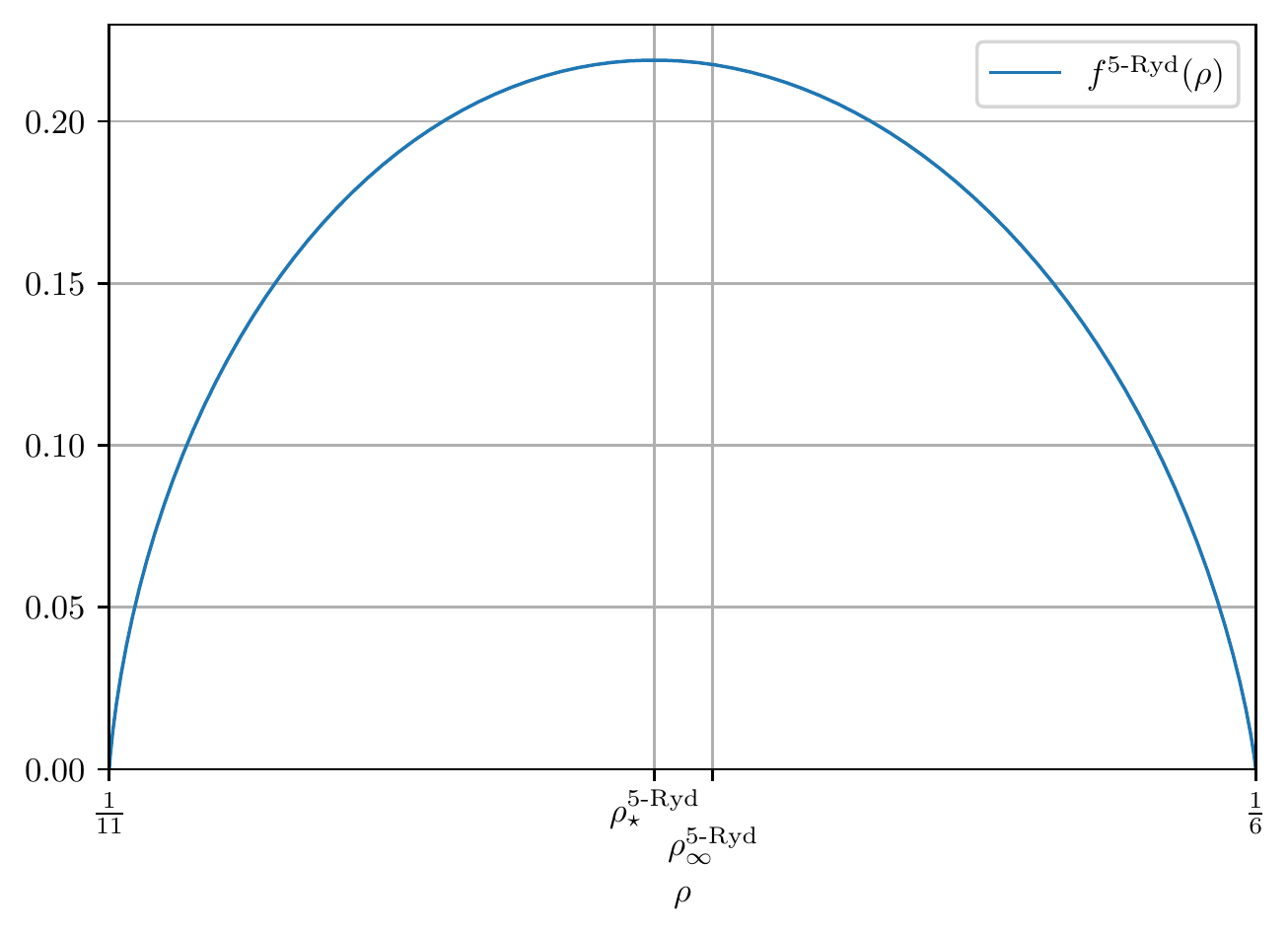}
	\end{subfigure} \\
	\begin{subfigure}{.48\linewidth}\centering
		\includegraphics[width=\linewidth]{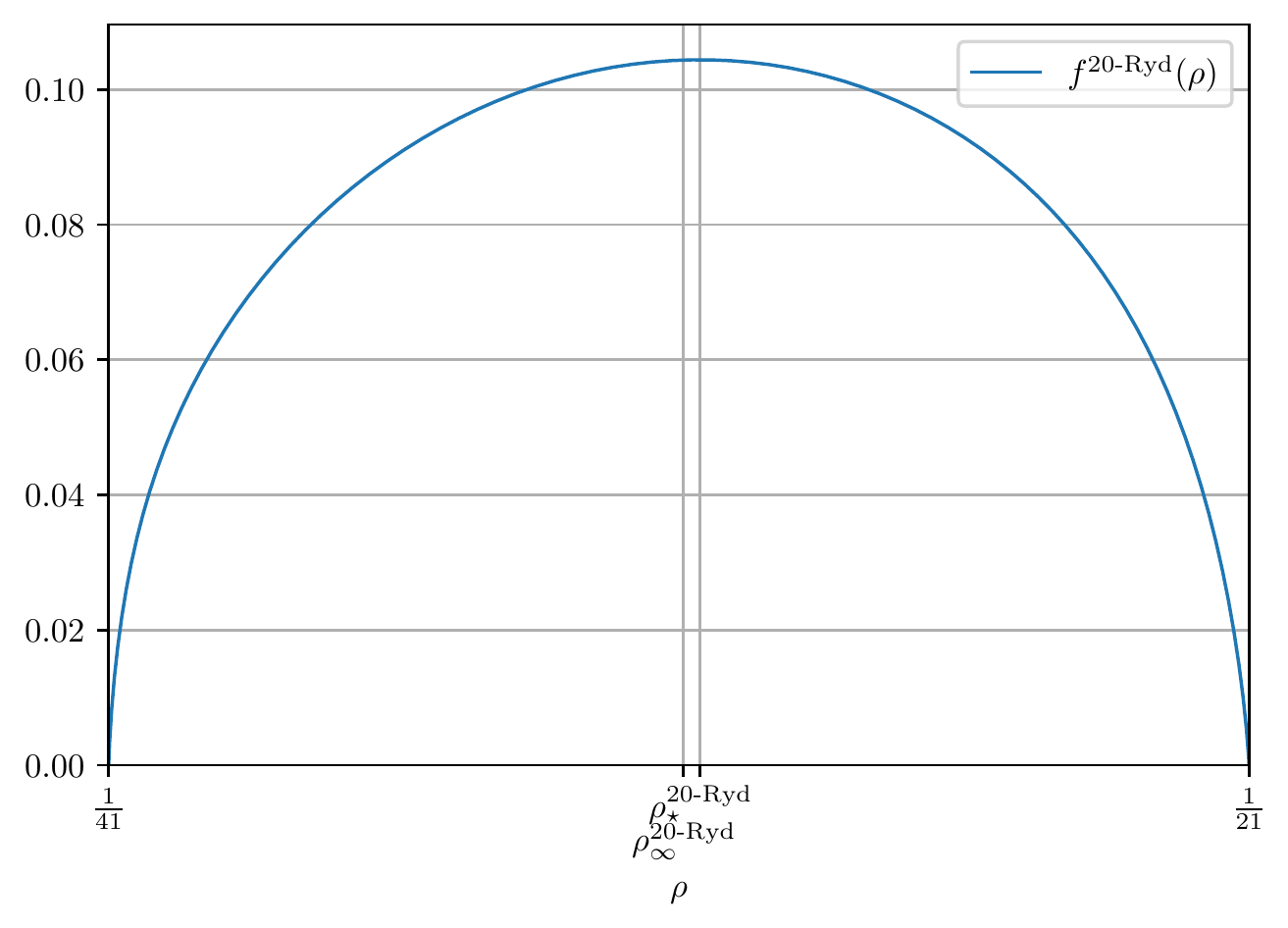}
	\end{subfigure}
	\begin{subfigure}{.48\linewidth}\centering
		\includegraphics[width=\linewidth]{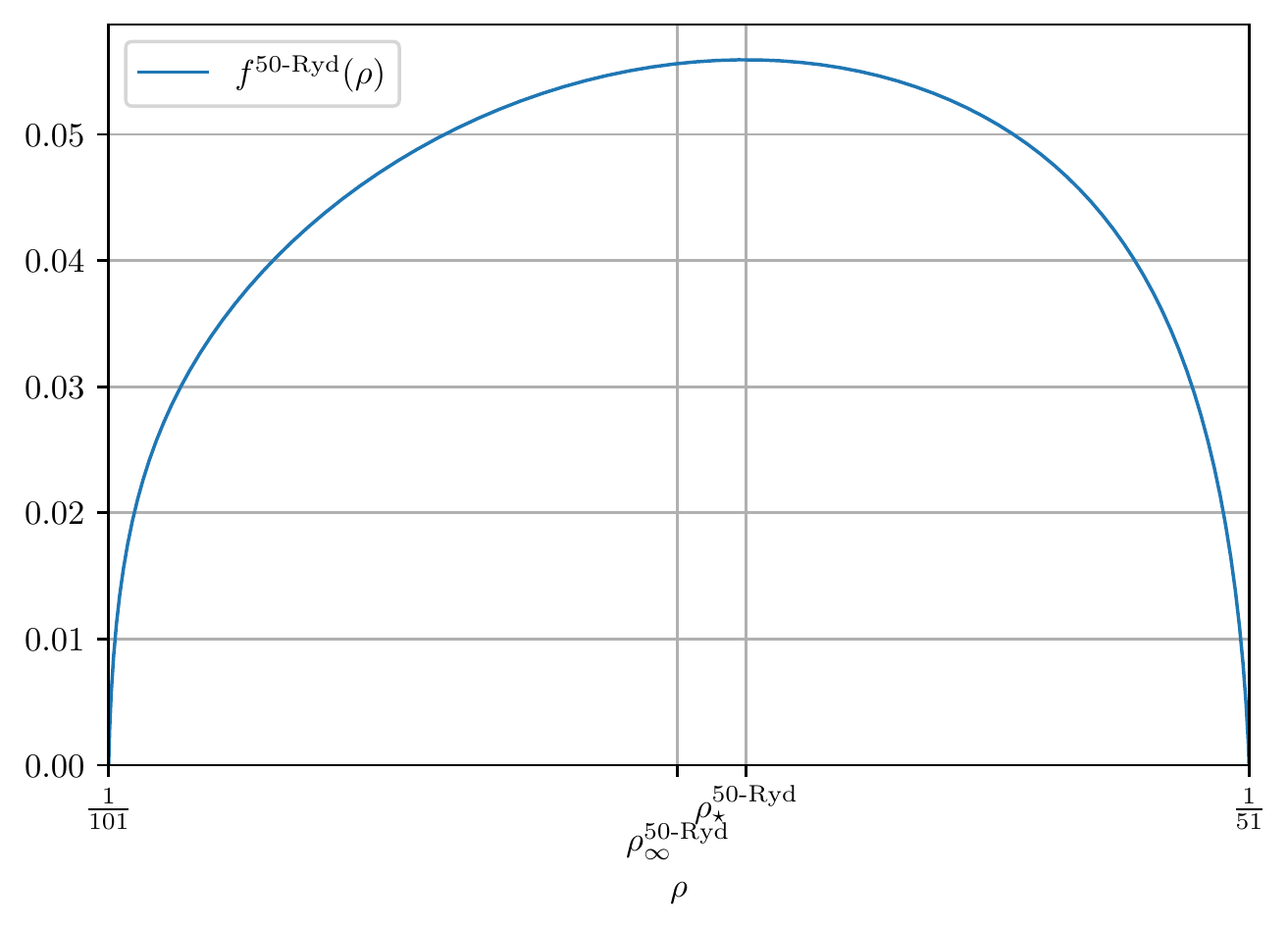}
	\end{subfigure}
	\caption{Complexity function of Rydberg atom model with blockade range $b$, for $b\in\{1,5,20,50\}$. Also plotted in each graph are the equilibrium density $\rho_\star^{b\text{-Ryd}}$ and the jamming density $\rho_\infty^{b\text{-Ryd}}$.}\label{fig:delta_Rydberg_separate}
\end{figure}
This violation of Edwards flatness hypothesis is even more pronounced when one inspects the asymptotics of the two sequences more closely. In Figure \ref{fig:delta_Rydberg_times_b} we see the graph of quantities $b\cdot\rho_{\star}^{b\text{-Ryd}}$ and $b\cdot\rho_{\infty}^{b\text{-Ryd}}$.
\begin{figure}
	\includegraphics[width=.6\linewidth]{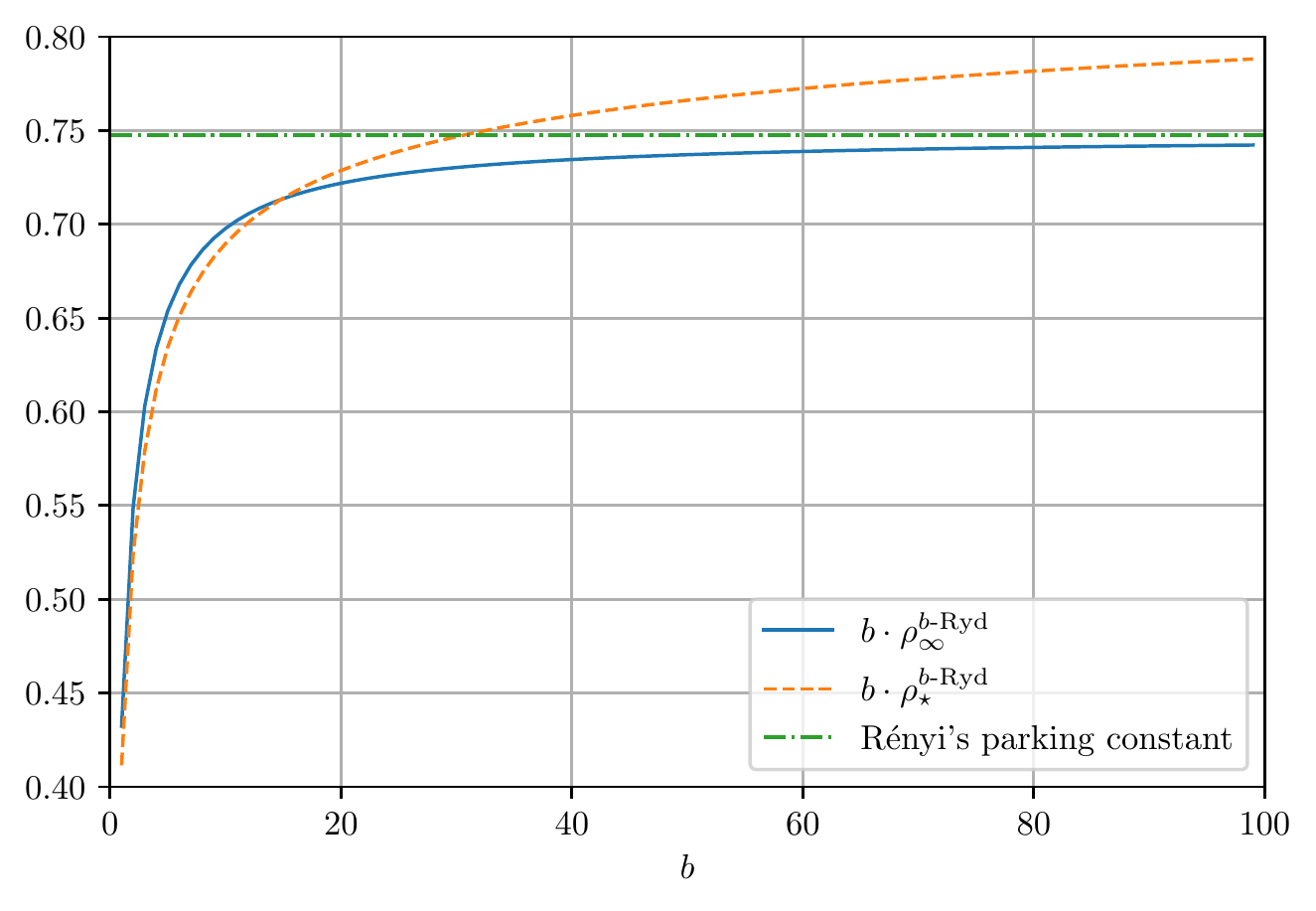}
	\caption{Comparison of $b\cdot\rho_\star^{b\text{-Ryd}}$ and $b\cdot\rho_\infty^{b\text{-Ryd}}$ for $1\le b \le 99$.}\label{fig:delta_Rydberg_times_b}
\end{figure}
It can be shown that these two sequences approach different constants as $b$ grows large
\begin{equation}\label{eq:limits}
\begin{split}
	\lim_{b\to\infty} b\cdot\rho_{\infty}^{b\text{-Ryd}} &= \int_0^\infty \exp\left[-2\int_0^{y} \frac{1-e^{-x}}{x} dx\right]dy=0.7475979202\dots\\
	\lim_{b\to\infty} b\cdot\rho_{\star}^{b\text{-Ryd}} &= 1.
\end{split}
\end{equation}
The constant appearing in the first limit is known as R\'enyi's parking constant \cite{58_renyi1958one}. Both of these two limits are easier to understand in the context of irreversible deposition of $k$-mers. We deal with the $k$-mer deposition model in the following section where we revisit those limits.

The calculation below, showing how to obtain the first limit in \eqref{eq:limits}, and which we provide for completeness, appears in \cite{54_gonzalez1974cooperative}. First note
\begin{multline*}
	\sum_{j=1}^{b} \frac{1-y^j}{j}
	= \sum_{j=1}^{b} \int_y^1 t^{j-1}\,dt
	= \int_y^1 \sum_{j=1}^{b} t^{j-1}\,dt
	= \int_y^1 \frac{1-t^b}{1-t}\,dt \\
	= \begin{bmatrix}x=b(1-t)\\dx=-b\,dt\end{bmatrix}
	= \int_0^{b(1-y)} \frac{1-(1-\frac{x}{b})^b}{x}\,dx,
\end{multline*}
and therefore
\begin{align*}
	b\cdot\rho_{\infty}^{b\text{-Ryd}}
	&= b \int_0^1 \exp\left[-2\sum_{j=1}^{b} \frac{1-y^j}{j} \right]\,dy\\
	&= b \int_0^1 \exp\left[-2 \int_0^{b(1-y)} \frac{1-(1-\frac{x}{b})^b}{x}\,dx \right]\,dy\\
	&= \begin{bmatrix}\tilde{y}=b(1-y)\\d\tilde{y}=-b\,dy\end{bmatrix}
	= \int_0^b \exp\left[-2 \int_0^{\tilde{y}} \frac{1-(1-\frac{x}{b})^b}{x}\,dx \right]\,d\tilde{y}.
\end{align*}
The dominated convergence theorem now implies
$$
	\lim_{b\to\infty}	b\cdot\rho_{\infty}^{b\text{-Ryd}}
	= \int_0^\infty \exp\left[-2\int_0^{y} \frac{1-e^{-x}}{x} dx\right]dy
	= 0.7475979202\dots
$$

Before we calculate the second limit in \eqref{eq:limits}, we give a characterization of the value $\rho_{\star}^{b\text{-Ryd}}$ in terms of a root of a certain polynomial. Compare this with the same results obtained by Došlić \cite[discussion after Theorem
2.10]{Doslic} and Krapivsky--Luck \cite[(3.4), (3.14) and (6.6)]{KL}.
\begin{theorem}\label{tm:rhostar}
	The value $\rho_{\star}^{b\text{-Ryd}}$, at which the complexity of the Rydberg atom model with blockade range $b$, given in Theorem \ref{tm:cmplx}, attains its maximum, can be calculated as
	\begin{equation}\label{eq:rhostar}
	\rho_{\star}^{b\text{-Ryd}} = \frac{(1 - z) (1 - z^{b+1})}{1 + b - b z - 2 z^{b+1} - 2 b z^{b+1} + z^{b+2} + 2 b z^{b+2}},
	\end{equation}
	where $z$ is the unique root of the polynomial
	$$z^{2b+1}+\dots+z^{b+2}+z^{b+1}-1,$$
	on the interval $0<z<1$.
\end{theorem}
\begin{proof}
	We seek to find the density $\frac{1}{2b+1}< \rho_{\star}^{b\text{-Ryd}}<\frac{1}{b+1}$ at which the complexity $f=f^{b\text{-Ryd}}$ in Theorem \ref{tm:cmplx} attains its maximum. Again, we employ the Lagrangian function method by setting
	\begin{align*}
	\mathcal{L}(\rho,z; \lambda) &= \rho\left[-\ln \frac{1-z}{1-z^{b+1}} - \left(\frac{1}{\rho}-(b+1)\right)\ln z\right] - \lambda \sum_{i=0}^{b} \left(i+b+1-\frac{1}{\rho}\right)z^i\\
	&= \rho\ln \frac{1-z^{b+1}}{1-z} - \left(1-\rho(b+1)\right)\ln z - \lambda \sum_{i=0}^{b} \left(i+b+1-\frac{1}{\rho}\right)z^i.
	\end{align*}
	The stationary points of this function solve the following system
	\begin{gather*}
		\ln \frac{1-z^{b+1}}{1-z} +(b+1)\ln z -\frac{\lambda}{\rho^2}\cdot\frac{1-z^{b+1}}{1-z} = 0 \\
		 -\frac{\rho (b+1)z^b}{1-z^{b+1}} + \frac{\rho}{1-z}- \frac{\left(1-\rho(b+1)\right)}{z} - \lambda \sum_{i=1}^{b} i\left(i+b+1-\frac{1}{\rho}\right)z^{i-1} = 0\\
		\sum_{i=0}^{b} \left(i+b+1-\frac{1}{\rho}\right)z^i = 0.
	\end{gather*}
	Using standard summation formulas, as in \eqref{eq:zPoly}, it is possible to express $\rho$ from the third equation as
	$$\rho = \frac{(1 - z) (1 - z^{b+1})}{1 + b - b z - 2 z^{b+1} -
	2 b z^{b+1} + z^{b+2} + 2 b z^{b+2}}.$$
	Plugging this into the second equation gives
	$$0=\lambda \sum_{i=1}^{b} i\left(i+b+1-\frac{1}{\rho}\right)z^{i-1}.$$
	From here, we conclude $\lambda=0$.
	Finally, from the first equation we get
	$$\lambda = \rho^2 \frac{1-z}{1-z^{b+1}} \ln\frac{z^{b+1}(1-z^{b+1})}{1-z}$$
	and, combining this with $\lambda=0$, gives
	$$\ln\frac{z^{b+1}(1-z^{b+1})}{1-z} = 0,$$
	or
	$${z^{b+1}(1-z^{b+1})}={1-z}.$$
	We know from Theorem \ref{tm:cmplx} that $z\neq1$, so we can rewrite this equation as
	$$z^{2b+1}+\dots+z^{b+2}+z^{b+1}-1=0.$$
	Clearly, there is a unique $0<z<1$ solving this equation, and the corresponding
	$$\rho_{\star}^{b\text{-Ryd}} = \frac{(1 - z) (1 - z^{b+1})}{1 + b - b z - 2 z^{b+1} - 2 b z^{b+1} + z^{b+2} + 2 b z^{b+2}}$$ is the density at which the complexity in the Rydberg atom model with blockade range $b$ is the largest.
\end{proof}

The previous theorem can be used to give a proof of the second limit in \eqref{eq:limits}.
\begin{corollary}\label{cor:limes_bRydberg}
	$$\lim_{b\to\infty} b\cdot\rho_{\star}^{b\text{-Ryd}}=1.$$
\end{corollary}
\begin{proof}
	Since $0<z=z(b)<1$ solves the equation
	\begin{equation}\label{eq:zb1}
		\frac{z^{b+1}(1-z^{b+1})}{1-z} = z^{2b+1}+\dots+z^{b+2}+z^{b+1} =1
	\end{equation}
	it follows
	$$bz^{2b+1}<1<bz^{b+1}$$ and therefore
	$$\lim_{b\to\infty} z^{2b+1} = 0.$$
	Multiplying by $z$ and taking square root we also get
	$$\lim_{b\to\infty} z^{b+1} = 0.$$
	Finally, letting $b\to\infty$ in the identity $z^{b+1}(1-z^{b+1})=1-z$, gives
	$$\lim_{b\to\infty} z = 1.$$
	Note that
	$$b\cdot\rho_{\star}^{b\text{-Ryd}} = \frac{b(1 - z) (1 - z^{b+1})}{1 + b(1-z)[1-2z^{b+1}] - 2 z^{b+1} + z^{b+2} }$$
	so in order to get $\lim_{b\to\infty}b\cdot\rho_{\star}^{b\text{-Ryd}}=1$, it suffices to show $\lim_{b\to\infty} b(1-z)=\infty$. To see this, note that from \eqref{eq:zb1} it follows
	$$(b+1)\ln z  = \ln (1-z) - \ln (1-z^{b+1})$$
	and hence
	$$\lim_{b\to\infty} (b+1)(1-z) =  \lim_{b\to\infty} \frac{1-z}{\ln z} \cdot\left[\ln (1-z) - \ln (1-z^{b+1})\right] = -1\cdot\left[-\infty -0\right] = +\infty$$
	which completes the argument.
\end{proof}

\section{Complexity function of jammed configurations for irreversible deposition of \texorpdfstring{$k$}{k}-mers}
\label{sectioncomplexitykmers}
It is easy to see that the Rydberg atom model with blockade range $b$ is, up to scaling all densities by a factor $b+1$, equivalent to the irreversible deposition of $k$-mers model where $k=b+1$. As an immediate consequence of Theorem \ref{tm:cmplx} we get the complexity of jammed configurations for irreversible deposition of $k$-mers.
\begin{corollary}\label{cor:cmplx}
	For $k\in\bbN$, $k>1$, the complexity function of jammed configurations for irreversible deposition of $k$-mers is
	$$f(\rho)  = \begin{cases}
	\frac{\rho}{k}\left[-\ln \frac{1-z}{1-z^{k}} - \left(\frac{k}{\rho}-k\right)\ln z\right], &\text{ if } \frac{k}{2k-1}<\rho \le 1,\\
	0, &\text{ otherwise,}
	\end{cases}$$
	where $z\ge 0$ is a real root of the polynomial
	\begin{equation*}
	\sum_{i=0}^{k-1} \left(i+k-\frac{k}{\rho}\right)z^i
	\end{equation*}
	for which the expression $f(\rho)$ is the largest.
\end{corollary}

\begin{figure}
	\includegraphics[width=\linewidth]{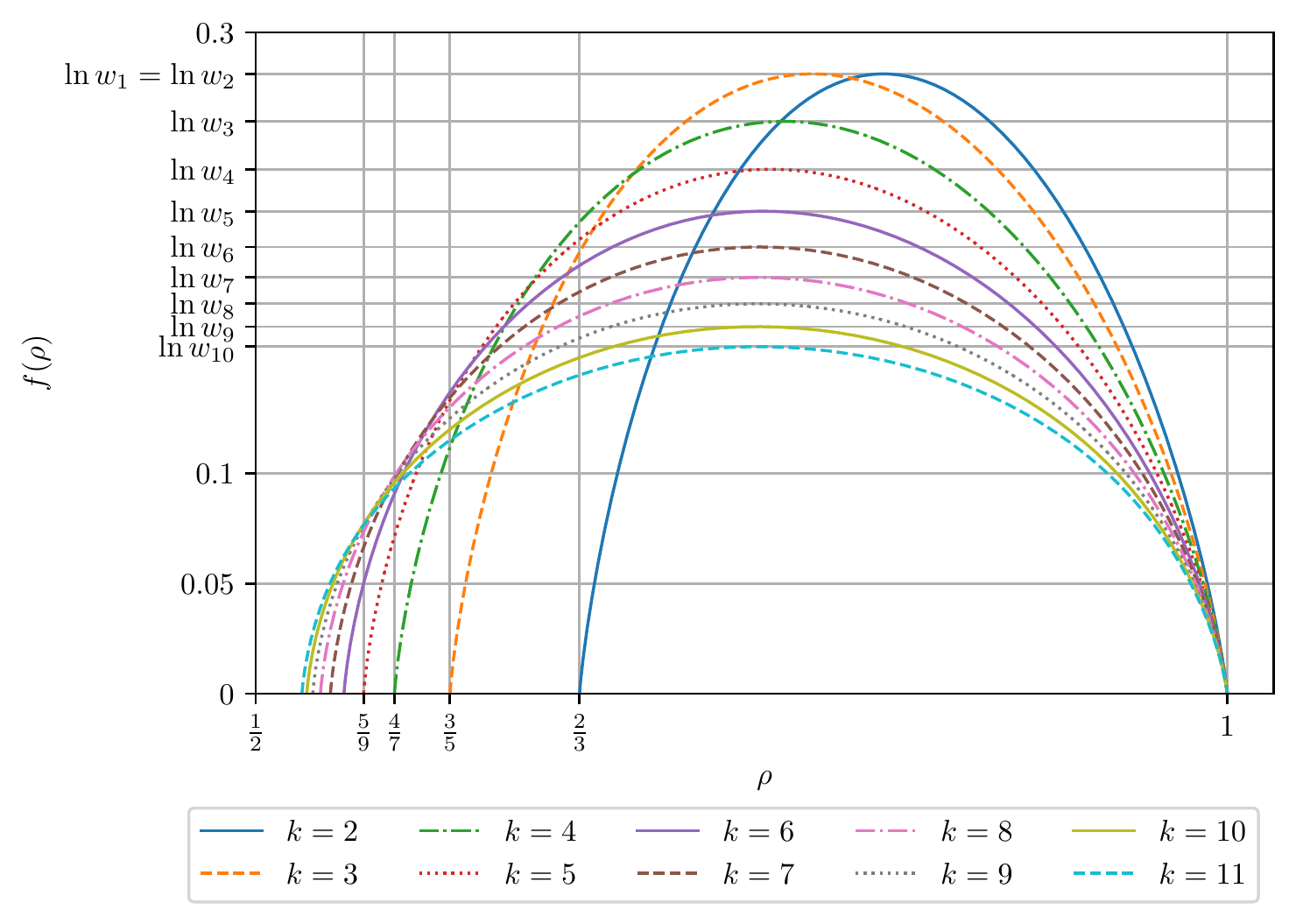}
	\caption{The complexity function of jammed configurations for irreversible deposition of $k$-mers, for $2\le k\le 11$.}\label{fig:kmer_all}
\end{figure}

Figure \ref{fig:kmer_all} shows the complexity function for all $2 \le k\le 11$.
Note that the support of the complexity function is now contained in the
interval $[1/2,1]$. In Figure \ref{fig:delta_all} we compare the equilibrium density $\rho_\star^{k\text{-mer}}$ and the jamming density $\rho_\infty^{k\text{-mer}}$, for $2\le k \le 100$. In this model it is even more obvious that the Edwards hypothesis is violated. The limits of these two sequences as $k$ grows large are
\begin{equation}\label{eq:k-mer-limits}
	\begin{aligned}
\lim_{k\to\infty} \rho_{\infty}^{k\text{-mer}} &= \int_0^\infty \exp\left[-2\int_0^{y} \frac{1-e^{-x}}{x} dx\right]dy=0.7475979202\dots \\
\lim_{k\to\infty} \rho_{\star}^{k\text{-mer}} &= 1.
\end{aligned}
\end{equation}
Note that these limits are equivalent to those in \eqref{eq:limits}. The convergence of jamming limits of $k$-mer deposition models (as $k$ grows to infinity) to the R\'{e}nyi's parking constant is discussed in \cite[\S 7.1]{30_krapivsky2010kinetic}.

Clearly, the second limit from \eqref{eq:k-mer-limits} follows from Corollary \ref{cor:limes_bRydberg} as $\rho_{\star}^{k\text{-mer}} = k\cdot\rho_{\star}^{b\text{-Ryd}}$ for $b=k-1$. Below, we provide a direct alternative proof of this fact.
\begin{theorem}
\begin{align*}
\lim_{k\to\infty} \rho_{\star}^{k\text{-mer}} &= 1.
\end{align*}
\end{theorem}
\begin{proof}
The quantity we are interested in, $\rho_{\star}^{k\text{-mer}}$, is
equivalent to the quantity called the {\em efficiency} $\varepsilon (k)$ in
the context of packing $P_k$ into $P_n$. It was shown in \cite{Doslic} that
the efficiency is determined by the smallest singularity $w_k$ of the
generating function $F_k(1,y)$, i.e., by the smallest zero of its denominator.
Hence we start by setting $x=1$ into the rightmost expression in
\eqref{eq:BGF_k-mers},
$$F_k(1,y) = \frac{1-y^k}{1-y-y^k-y^{2k}}=\frac{\frac{1-y^k}{1-y}}{1-y^k \frac{1-y^k}{1-y}}.$$
We rewrite its denominator as $1-q_k(y)$, where $q_k(y) =q^k \frac{1-y^k}{1-y}$,
and denote the smallest solution of equation $q_k(y) = 1$ by $w_k$.
This equation has only one positive solution, since $q_k(0) = 0$,
$q_k(1) = k > 1$ for large $k$ and $q_k'(y) > 0$ for all $y > 0$. Moreover,
the same reasoning provides a better lower bound for $w_k$, since
$q_k (\frac{1}{2}) = 2^{(1-k)}(1-2^{-k}) < 1$. Hence $1/2 < w_k < 1$.

Consider now the expression
$$\varepsilon (k) = \rho_{\star}^{k\text{-mer}} = \frac{k}{w_k q_k ' (x)}$$
derived in \cite{Doslic}. First we rewrite $q_k ' (w_k)$ as
$$q_k'(x) = x^k \frac{1-x^k}{1-x} \left [ \frac{2k}{x} - \frac{k}{x(1-x^k)}
+ \frac{1}{1-x} \right ] .$$
After plugging in $x = w_k$, the term outside the brackets becomes equal to one,
and by multiplying through by $w_k$ we arrive at
$$w_k q_k ' (w_k) = \left ( 2 - \frac{1}{1-w_k^k} \right ) k + \frac{w_k}{1-w_k}.$$
We are seeking upper bounds to the right-hand side. The first term is bounded
from above by $k$, since the expression in parentheses cannot exceed one.
It remains to bound the second term. As mentioned before, $w_k$ is the only
positive solution of the equation $1-q_k(x) = 0$. We claim that, for a given
(large) positive $a$, $w_k < 1-\frac{a}{k}$ for large enough $k$. So let us
suppose otherwise, that for a given $a > 0$, $w_k > 1-\frac{a}{k}$ is valid
for all $k$. It means that the function $1-q_k(x)$ has a positive value for
$x = 1-\frac{a}{k}$. By evaluating both sides, we obtain that
$$\left ( 1-\frac{a}{k} \right ) ^k - \left ( 1-\frac{a}{k} \right ) ^{2k} < \frac{a}{k}$$
is valid for all $k$.
This is a contradiction, since the left-hand side has a positive limit,
$e^{-a} - e^{-2a} > 0$, while the right-hand side tends to zero as $k$
tends to infinity. Hence,
$w_k < 1-\frac{a}{k}$ for large enough $k$. Now the second term can be
bounded from above by $\frac{a}{k}$, and the whole expression for
$w_k q_k'(w_k)$ is bounded from above by $\frac{a+1}{a} k$. Since $a$ can be
taken arbitrarily large, it means that the reciprocal value of $w_k q_k'(w_k)$,
which is equal to our $\rho_{\star}^{k\text{-mer}}$, comes arbitrarily close
to one, and our claim follows.
\end{proof}
The convergence is quite slow, most likely logarithmic. We note another unusual thing in Figure \ref{fig:delta_all}. The equilibrium
density $\rho_\star^{k\text{-mer}}$ attains the minimum value for $k=9$. The
interpretation being that the polymers of length $9$ pack the least efficiently
of all polymers assuming the equilibrium model. This phenomenon was previously
observed in \cite{Doslic}.

\begin{figure}
	\includegraphics[width=.6\linewidth]{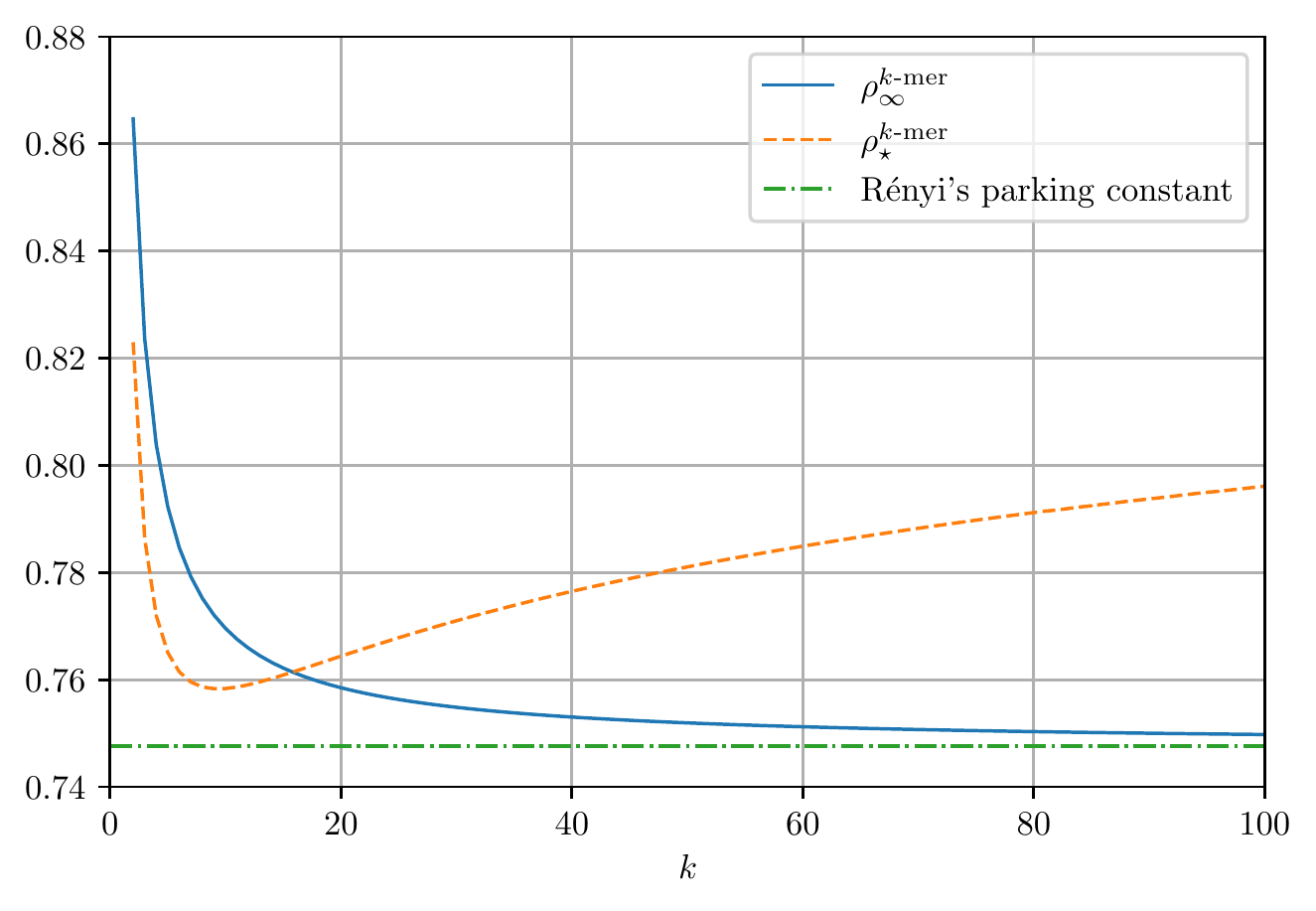}
	\caption{Comparison of $\rho_\star^{k\text{-mer}}$ and $\rho_\infty^{k\text{-mer}}$ for $2\le k \le 100$.}\label{fig:delta_all}
\end{figure}

\section{Conclusions}\label{sec:concluding}
In this paper we have computed the complexity function (or configurational
entropy) of jammed configurations of Rydberg atoms with a given blockade
range on a one-dimensional lattice. We employed a purely combinatorial method
which allowed us to compute the complexity function by solving a constrained
optimization problem. Along the way we have explored and elucidated numerous
connections between the considered problem and other models, such as, e.g.,
the random sequential adsorption and packings of blocks of a given length
into one-dimensional lattices. In most cases, we have not followed those links
very far. We believe that many interesting results could be obtained by
deeper investigations of those connections. As an example, we mention here
that explicit expressions for the number of maximal packings of given size
from reference \cite{Doslic} could be directly translated into expressions
for the number of jammed configurations of Rydberg atoms. By the same reasoning
one can show that the total number of all jammed configurations of $N$
Rydberg atoms with blockade range $b$ on all one-dimensional lattices
is given by $(b+1)^{N+1}$.

The methods employed here could be easily adapted for other one-dimensional
structures with low connectivity such as, e.g., cactus chains. Another class
of promising structures could be various simple graphs decorated by addition
of certain number of vertices of degree one to each of their vertices.

Similar problems were considered under various guises also for finite portions
of rectangular lattices, mostly for narrow strips of varying length. Among the
best known problems of this type are the so-called unfriendly seating
arrangements. See \cite{chern,georgiu} for their history and some recent
developments. To the same class belong the problems concerned with privacy,
such as the ones considered in \cite{kranakis}. All cited references were
concerned with one-dimensional lattices and/or narrow strips in the square
grid, mostly with ladders. It would be interesting to consider those problems
in finite portions of the regular hexagonal lattice.

Another interesting thing to do would be to study the behavior (and the
difference) of $\rho_{\infty}$ and $\rho_{\star}$ for different
lattices/substrates. In other words, to investigate the difference between
the jamming limit of dynamical models and the most probable densities in
the equilibrium models. A drastic example is presented by the expected
density of independent sets in stars: there are exactly two maximal independent
sets in $S_n = K_{1,n-1}$, one of them with size 1 and the other with size
$n-1$. If both of them are equally probable, the expected size is $n/2$.
Under dynamical model, however, the smaller one is much less probable than
the bigger one, and the expected size is $\frac{1}{n} + \frac{n-1}{n}(n-1)
= n-2 + \frac{2}{n}$. It would be interesting to know more about such
differences and to know for which classes of graphs they are extremal.

Our final remark is that the jammed configurations of Rydberg atoms with a
given blockade range $b$ are known as maximal $b$-independent sets in the
language of graph theory. It might be worth investigating to what extent
can similar problems be formulated also in terms of $b$-dominance in graphs.
\section*{Acknowledgments}
\noindent
It is a pleasure for us to thank Jean-Marc Luck and Pavel Krapivsky for
fruitful exchanges during the concomitant elaboration of their preprint
\cite{KL} and of the present work.
Partial support of Slovenian ARRS (Grant no. J1-3002) is gratefully
acknowledged by T. Do\v{s}li\'c.

\bibliographystyle{bababbrv-fl}
\bibliography{literature}

\end{document}